\documentclass[12pt]{article}
\usepackage{amsmath,amssymb,color,amsthm,amsfonts}
\usepackage{mathtools}
\usepackage{hyperref}
\usepackage[nottoc,numbib]{tocbibind}
\usepackage[style=alphabetic]{biblatex}
\usepackage{tikz-cd}

\newcommand{\on}{\operatorname}

\DeclareMathOperator{\Sym}{Sym}

\DeclareMathOperator{\Hom}{Hom}
\DeclareMathOperator{\Spec}{Spec}

\DeclareMathOperator{\LocSys}{LocSys}

\DeclareMathOperator{\TL}{TL}

\newtheorem{theorem}{Theorem}[section]
\newtheorem{lemma}[theorem]{Lemma}
\newtheorem{proposition}[theorem]{Proposition}
\newtheorem{corollary}[theorem]{Corollary}
\newtheorem{definition}[theorem]{Definition}

\theoremstyle{remark}
\newtheorem{example}[theorem]{Example}
\newtheorem{remark}[theorem]{Remark}

\title{On Moy-Prasad quotients over Laurent series fields}
\author{David Yang}

\begin{document}

\maketitle

\begin{abstract}
Let $k$ be an algebraically closed field and $G$ a connected reductive group over $k((t))$ satisfying some conditions. We define a stratification by conjugacy classes of twisted Levi subgroups of $G$ on each Moy-Prasad quotient $\mathfrak{k}_{x,r}/\mathfrak{k}_{x,r+}$ of $G$. We then calculate the strata in terms of the associated twisted Levi subgroups. This calculation is necessary for several followup papers on the local geometric Langlands program.
\end{abstract}

\tableofcontents

\section{Introduction}

This paper is an input into several followup papers on the local geometric Langlands program, some joint with other authors, including Gurbir Dhillon and Yakov Varshavsky. To motivate our main result, we will sketch some ideas from those papers.

For the entirety of this paper, fix an algebraically closed field $k$ and a connected reductive group $G$ over $k((t)).$ We will need to impose characteristic restrictions on $k$ later, but let us temporarily ignore this point. The local geometric Langlands program aims to classify categorical representations of the $k$-indscheme $G((t))$, whose functor of points is defined by $\Hom_k(\Spec A,G((t)))\cong\Hom_{k((t))}(\Spec A((t)),G((t))).$ (In many previous works, $G$ is assumed to be split, but as we will see later, it is essential to generalize to the non-split case as well.) One of the main results of our followup papers will be that all categorical representations of $G((t))$ can be built up from representations obtained by a (nontrivial) induction process from twisted Levi subgroups (i.e., a subgroup of $G$ which becomes a Levi subgroup after passing to the algebraic closure of $k((t))$).

If $k$ were a finite field instead of an algebraically closed field, Moy and Prasad \cite{MP1} assigned to every irreducible vector space representation of $G((t))$ an unrefined minimal K-type, which consists of a point $x$ in the building of $G$ and a character of a vector space $\mathfrak{k}_{x,r}/\mathfrak{k}_{x,r+}$ defined by Moy-Prasad. For $k$ of large enough characteristic, the dual of this vector space can be identified with another quotient $\mathfrak{k}_{x,-r}/\mathfrak{k}_{x,-r+}$. In a previous paper \cite{yangMP}, we generalized this theory to the geometric setting mentioned above. There is no notion of irreducible categorical representation, but we showed that, in some sense, every categorical representation can be built out of representations for which one can associate a family of unrefined minimal K-types.

As a first step for defining the induction process mentioned above, we will need to relate unrefined minimal K-types for $G$ to unrefined minimal K-types for twisted Levis. This is the goal of the current paper. Let us now explain what we do in greater detail.

In Section \ref{s:prelim} we will recall various preliminaries from Bruhat-Tits theory and Moy-Prasad theory. For the reader unaccustomed to non-split groups, we will also explain how to work with them explicitly in our setting. 

In Section \ref{s:stratification}, we will build up to defining a stratification $(\mathfrak{k}_{x,r}/\mathfrak{k}_{x,r+})^{\geq G_0}$ by conjugacy classes of twisted Levis $G_0$ on each vector space $\mathfrak{k}_{x,r}/\mathfrak{k}_{x,r+}.$ Recall that an unrefined minimal K-type contained the information of a point in some $\mathfrak{k}_{x,-r}/\mathfrak{k}_{x,-r+}$. Our stratification tells us which twisted Levi (up to conjugacy) our minimal K-type should be induced from. The formulation of this principle is the main result of this paper, Theorem \ref{basecase}. It describes the variety $(\mathfrak{k}_{x,r}/\mathfrak{k}_{x,r+})^{=G_0}$ in terms of varieties $(\mathfrak{k}_{x,r}^{G_s}/\mathfrak{k}_{x,r+}^{G_s})^{=G_s}$ for conjugates $G_s$ of $G_0$ such that $x$ lies in the (extended) building of $G_s$.

In Section \ref{s:proof} we give the proof of Theorem \ref{basecase}.

\subsection{Notation and conventions}

We will also fix an algebraic closure $\overline{k((t))}$ of $k((t))$ and an embedding of the field of Puiseux series $\cup k((t^{\frac{1}{n}}))$ into $\overline{k((t))}$. We will need the assumption that $G$ splits over a tamely ramified extension of $k((t))$ and that the characteristic of $k$ does not divide the order of the Weyl group of $G\times_{\Spec k((t))}\Spec\overline{k((t))}$. (Note that this is the same as Assumption 2.1 of \cite{tametypes}.)

\subsection{Acknowledgements}

We would like to thank Ekaterina Bogdanova, Lin Chen, Gurbir Dhillon, Dennis Gaitsgory, Kevin Lin, Sam Raskin, and Yakov Varshavsky for helpful conversations. 

\section{Groups over $k((t))$}\label{s:prelim}

\subsection{Quasi-split groups and twisted loop algebras}

In our previous paper \cite{yangMP}, we studied the Moy-Prasad theory of split groups over $k((t))$. Our main motivation remains the split case, but the techniques of this paper (and its follow-ups) necessitate working in the generality of non-split groups.

As observed in \cite{pappasrapoport}, the group $G((t))$ of $k((t))$-points of $G$ can still be identified with a twisted loop group. To see this, consider a tame field extension $k((t^{\frac{1}{n}}))/k((t))$ over which $G$ splits. Then, the Galois group $\on{Gal}(k((t^{\frac{1}{n}}))/k((t)))$ acts on $G((t^{\frac{1}{n}}))$ with fixed points $G((t))$.

Let $H$ be the split form over $k$ of $G$, and fix a pinning $(H,T_H,\Delta_H,(e_{\alpha,H}))$ of $H$ (see e.g., Milne \cite{milneRG}, 19.49). Then $G$ can be described as the twist of $H\times_{\Spec k}\Spec k((t))$ by a Galois 1-cocycle into the group $\on{Aut}_{k((t^{\frac{1}{n}}))}(H\times_{\Spec k}\Spec k((t^{\frac{1}{n}})))$ of group automorphisms. By Steinberg's theorem \cite{steinberg}, $G$ is quasi-split and so this Galois 1-cocycle can be chosen to take values in $\on{Aut}_{k((t))}((H,T_H,\Delta_H,(e_{\alpha,H}))\times_{\Spec k}\Spec k((t^{\frac{1}{n}})))\cong \on{Aut}_{k}(H,T_H,\Delta_H,(e_{\alpha,H}))$.

Explicitly, choose a generator $a$ of $\on{Gal}(k((t^{\frac{1}{n}}))/k((t)))$, which will send $t^{\frac{1}{n}}$ to $\zeta t^{\frac{1}{n}}$ for $\zeta$ some primitive $n$th root of unity. The above 1-cocycle sends $a$ to an element of $\on{Aut}_{k}(H,T_H,\Delta_H,(e_{\alpha,H})),$ which gives an automorphism $\sigma$ of $H$. We can sum up this discussion with the following lemma. Let $B_H$ be the Borel subgroup containing $T_H$ determined by the choice of $\Delta_H$. 

\begin{lemma}\label{twistedaffinegroup}
If $H$ is the split form over $k$ of $G$, there exists a positive integer $n$ not divisible by the characteristic of $k$, a generator $a:t^{\frac{1}{n}}\mapsto\zeta t^{\frac{1}{n}}$ of $\on{Gal}(k((t^{\frac{1}{n}}))/k((t)))$, an endomorphism $\sigma$ of $H$ preserving $T_H$ and $B_H$, and an isomorphism
\[
\Xi:G\times_{\on{Spec}k((t))}\on{Spec}k((t^{\frac{1}{n}}))\cong H\times_{\on{Spec}k}\on{Spec}k((t^{\frac{1}{n}}))
\]
which intertwines the action of $\on{id}\times a$ on the left and $\sigma\times a$ on the right. This induces an isomorphism between $G((t))$ and the fixed point ind-scheme of $\sigma\times a$ acting on $H((t^\frac{1}{n}))$.
\end{lemma}

\begin{remark}
In the statement of Lemma \ref{twistedaffinegroup}, we fix only a Borel subgroup and a maximal torus inside it, rather than an entire pinning. This weakening will be convenient in a sequel paper for discussing functoriality of $\LocSys$ for twisted Levi subgroups.
\end{remark}

At the level of Lie algebras, the above description gives the following. Let us abuse notation and denote the induced action of $\sigma$ on the Lie algebra $\mathfrak{h}$ again by $\sigma$. 

\begin{lemma}\label{twistedaffine}
There is an automorphism $\theta$ of $\mathfrak{h}((t^{\frac{1}{n}}))$ given by
\[
\theta(h\cdot t^{\frac{i}{n}})=\zeta^{i}\cdot\sigma(h)\cdot t^{\frac{i}{n}},
\]
and $\mathfrak{g}((t))$ can be identified with the fixed points of $\theta$ on $\mathfrak{h}((t^{\frac{1}{n}})$.
\end{lemma}

We note that the isomorphisms of Lemmas \ref{twistedaffinegroup} and \ref{twistedaffine} are not canonical. We will thus aim to make the statements of our main theorems independent of such a choice of isomorphism, but we will often find it convenient to fix a choice in the proofs.

\subsection{The Bruhat-Tits building}\label{ss:building}

The theory of the Bruhat-Tits building is deep and extensive, and here we can only give a cursory overview. For a more comprehensive treatment, we refer the reader to the recent book \cite{kalethaprasad} of Kaletha-Prasad.

Let $\mathcal{B}(G)$ be the Bruhat-Tits building of $G$. This is a topological space which also has a natural structures of a polysimplicial complex. It also carries a canonical action of the discrete group $G(k((t)))$ of $k((t))$-points of $G$. 

For any maximal split torus $T_G$ of $G$, there is a polysimplicial sub-complex $\mathcal{A}(T_G)$ of $\mathcal{B}(G)$, called the apartment of $\mathcal{B}(G)$ corresponding to $T_G.$ As a topological space, it is isomorphic to $X_*(T_G/(T_G\cap Z_G))\otimes\mathbb{R}$ (equipped with the Euclidean topology), where $Z_G$ is the center of $G$. We have an equality $\mathcal{B}(G)=G(k((t)))\cdot\mathcal{A}(T_G)$, which we will frequently use to reduce statements for $\mathcal{B}(G)$ to statements for $\mathcal{A}(T_G)$.

It is easiest to describe the polysimplicial structure of $\mathcal{A}(T_G)$ when $T_G$ is of a specific form. Fix $(n,a,\sigma,\Xi)$ as in Lemma \ref{twistedaffinegroup}. Let $T$ be the connected component of the identity in the fixed point locus $T_H^{\sigma}.$ Then $T\times_{\Spec k}\Spec k((t))$ becomes a maximal split torus of $G$ after choosing an isomorphism as in Lemma \ref{twistedaffinegroup}. Let us describe $\mathcal{A}(T\times_{\Spec k}\Spec k((t)))$ (henceforth abbreviated to $\mathcal{A}(T)$, though we warn the reader that this subspace of $\mathcal{B}(G)$ implicitly depends on the choice of $\Xi$ in Lemma \ref{twistedaffinegroup}.) As an affine space, $\mathcal{A}(T)$ is isomorphic to $X_*(T/(T\cap Z_H))\otimes\mathbb{R}$, where $Z_H$ is the center of $H$.

Diagonalizing $\mathfrak{g}((t))$ with respect to the $T$ action and the degree operator, we get the affine root decomposition

\[
\mathfrak{g}((t))\cong\displaystyle\bigoplus_{\substack{\alpha\in X^*(T)\\ i\in\frac{1}{n}\mathbb{Z}}}\mathfrak{h}_{\alpha}^it^i,
\]
where $h_{\alpha}$ is the root space corresponding to $\alpha$ and $h_{\alpha}^i$ is its $\zeta^{ni}$-eigenspace for $\sigma$. (To save ink, future sums of affine root spaces will implicitly be indexed by an element $\alpha\in X^*(T)$ and a number $i\in\frac{1}{n}\mathbb{Z}$.) Each non-zero affine root space $\mathfrak{h}_{\alpha}^it^i$ corresponds to a functional on $\mathcal{A}(T)$, given by $x\mapsto\langle x,\alpha\rangle+i.$ These functionals form an affine root system, so their zero sets are hyperplanes which induce a poly-simplicial decomposition of $\mathcal{A}(T)$ whose maximal facets are (the closures of) the connected components of the complement of the hyperplane arrangement. This gives the desired structure of poly-simplicial complex on $\mathcal{A}(T)$.

To each point $x\in\mathcal{B}(G)$, Bruhat-Tits associates a parahoric subgroup $P_x\subseteq G((t))$. When $x$ lies in the apartment $\mathcal{A}(T)$, $P_x$ has Lie algebra
\[
\mathfrak{p}_{x}=\displaystyle\bigoplus_{\langle\alpha,x\rangle+i\thinspace\geq\thinspace 0}\mathfrak{h}_{\alpha}^it^i
\]
and its unipotent radical $P_x^+$ has Lie algebra
\[
\mathfrak{p}_x^+=\displaystyle\bigoplus_{\langle\alpha,x\rangle+i\thinspace>\thinspace 0}\mathfrak{h}_{\alpha}^it^i.
\]

It is clear that for $x\in\mathcal{A}(T)$, $P_x$ depends only on the facet in which $x$ is contained. Conversely, knowledge of $P_x$ is sufficient to recover the point $x$. This will be more broadly true for the entire building. In other words, facets of $\mathcal{B}(G)$ are in bijection with parahoric subgroups of $G((t))$, and the inclusion relation on facets is the inverse of the inclusion relation on parahorics. In particular, vertices of $\mathcal{B}(G)$ correspond to maximal parahorics, and maximal simplicies of $\mathcal{B}(G)$ correspond to Iwahori subgroups of $G$. This bijection intertwines the action of $G(k((t)))$ on $\mathcal{B}(G)$ with the conjugation action on parahorics.

For any positive integer $m$, we can also consider the building $\mathcal{B}(G,k((t^{\frac{1}{m}})))$ associated to the algebraic group $G\times_{\Spec k}\Spec k((t^{\frac{1}{m}}))$ over $k((t^{\frac{1}{m}}))$. There is an action of $\on{Gal}(k((t^{\frac{1}{m}}))/k((t)))$ on $\mathcal{B}(G,k((t^{\frac{1}{m}})))$, and $\mathcal{B}(G)$ can be recovered as the fixed points of this Galois action. If we take $m=n$, then $T_H\times_{\Spec k}\Spec k((t^{\frac{1}{n}}))$ can be viewed a maximal torus inside $H\times_{\Spec k}\Spec k((t^{\frac{1}{n}}))\cong G\times_{\Spec k((t))}\Spec k((t^{\frac{1}{n}}))$, and the associated apartment (henceforth denoted $\mathcal{A}(T_H)$) in $\mathcal{B}(G,k((t^{\frac{1}{n}})))$ will contain $\mathcal{A}(T)$ as the set of Galois-fixed points.

Finally, let us also define the extended Bruhat-Tits building $\widetilde{\mathcal{B}}(G)$, which is defined to be the product of $\mathcal{B}(G)$ with the affine space $X_*(Z_G)\otimes\mathbb{R}$. The extended apartment $\widetilde{\mathcal{A}}(T)$ will then become isomorphic to $X_*(T)\otimes\mathbb{R}$. We can similarly define $\widetilde{\mathcal{B}}(G,k((t^{\frac{1}{m}})))$ and $\widetilde{\mathcal{A}}(T_H)\cong X_*(T_H)\otimes\mathbb{R}$. The extended building is technically not a poly-simplicial complex; however, it enjoys arguably better functoriality properties. 

The stabilizers of points in the extended building can be given a group scheme structure via the following lemma.

\begin{lemma}\label{stabilizerconstruction}
Let $x$ be a point of $\widetilde{\mathcal{B}}(G)$. Then there is a unique subgroup scheme $G^x\subseteq G((t))$ satisfying the two following properties.
\begin{enumerate}
\item The connected component of the identity in $G^x$ is equal to $P^x$.
\item The group of $k$-points $G^x(k)$ is the stabilizer inside $G(k((t)))$ of $x$.
\end{enumerate}
\end{lemma}

\begin{proof}
By Theorem 0.1 of \cite{pappasrapoport}, the set of $k$ points of the connected component of the identity in $G((t))$ can be identified with the kernel of the Kottwitz homomorphism $\kappa:G(k((t)))\rightarrow\pi_1(G)_I$ (see loc.cit. for the notation). Furthermore, by \cite{onparahoricsubgroups}, the intersection of the stabilizer of $x$ with the kernel of $\kappa$ is equal to $P_x(k)$. We can thus construct $G^x$ as a disjoint union of suitable translates of $P_x$, one for each connected component of $G((t))$ containing a $k$-point stabilizing $x$. To see that $G^x$ is a scheme (as opposed to an ind-scheme), it suffices to know that the number of translates is finite. But as the stabilizer inside $G(k((t)))$ of $x$ is bounded, its image under $\kappa$ is finite, as desired. Conversely, it is straightforward to see that any subgroup scheme satisfying the two properties of the lemma must agree with the one constructed above.
\end{proof}

\subsection{Moy-Prasad subgroups of quasi-split groups}\label{ss:moyprasad}

For every point $x$ in $\mathcal{B}(G)$, Moy and Prasad \cite{MP1}\footnote{Technically, \cite{MP1} works only with groups over non-archimedean local fields. However, the same arguments apply to the case of $k((t))$. This is made explicit in e.g. Chapter 13 of \cite{kalethaprasad}.} define a series of subgroups $K_{x,r}$ of $G((t))$. These subgroups satisfy $K_{gx,r}=gK_{x,r}g^{-1},$ so to fix the Moy-Prasad subgroups it suffices to specify them for $x$ in the apartment $\mathcal{A}(T)$, where $T$ is as in the previous subsection.

When $r=0$, the subgroups $K_{x,0}$ and $K_{x,0+}$ recover the subgroups $P_x$ and $P_x^+$. More generally, let $r$ be any real number. Then we can define the following subspaces of $\mathfrak{g}((t))$:
\[
\mathfrak{k}_{x,r}=\displaystyle\bigoplus_{\langle\alpha,x\rangle+i\thinspace\geq\thinspace r}\mathfrak{h}_{\alpha}^it^i
\]
and
\[
\mathfrak{k}_{x,r+}=\displaystyle\bigoplus_{\langle\alpha,x\rangle+i\thinspace>\thinspace r}\mathfrak{h}_{\alpha}^it^i.
\]
When $r$ is nonnegative, these subspaces are in fact Lie sub-algebras. As expected, for $r=0$, we recover $\mathfrak{p}_x$ and $\mathfrak{p}_x^+$. For positive $r$, both $\mathfrak{k}_{x,r}$ and $\mathfrak{k}_{x,r+}$ are contained in $\mathfrak{p}_x^+$, and as $P_x^+$ is pro-unipotent, we can exponentiate to get subgroups $K_{x,r}$ and $K_{x,r+}$. 

\begin{example}
Take $G$ to be split. A splitting gives rise to a canonical origin (denoted by $0$) inside $\mathcal{A}(T).$ Then for $n$ a nonnegative integer, $K_{0,n}$ is the $n$th congruence subgroup and $K_{0,n+}$ is the $(n+1)$th congruence subgroup.
\end{example}

The $\mathfrak{k}_{x,r}$ satisfy nice commutation properties. It is quickly checked that we have
\[
[\mathfrak{k}_{x,r},\mathfrak{k}_{x,s}]\subseteq\mathfrak{k}_{x,r+s}
\]
for any real numbers $r$ and $s$. When $r$ and $s$ are nonnegative, this can be integrated to the corresponding subgroups:
\[
[K_{x,r},K_{x,s}]\subseteq K_{x,r+s},
\]
and analogous relations for the $\mathfrak{k}_{x,r+}$ and $K_{x,r+}$ follow formally. A particularly important consequence is that the adjoint action of $P_x$ preserves all of the $\mathfrak{k}_{x,r}$, and so for $r\geq 0$, $K_{x,r}$ is a normal subgroup of $P_x$.

We briefly note that we can also define Moy-Prasad subgroups for $x'\in\widetilde{\mathcal{B}}(G)$ by setting $K_{x',r}=K_{x,r},$ where $x$ is the image of $x'$ in $\mathcal{B}(G).$

\subsection{Twisted Levi subgroups}

We will need a generalization of the class of Levi subgroups.

\begin{definition}
A \textbf{twisted Levi subgroup} of $G$ is a subgroup $G_0\subseteq G$ which becomes a Levi subgroup after passing to some extension of $k((t))$.
\end{definition}

\begin{definition}
Let $\TL(G)$ be the set of conjugacy classes of twisted Levi subgroups, equipped with the partial order where the conjugacy class of $G_0$ is less than or equal to the conjugacy class of $G_1$ iff $G_1$ contains a conjugate of $G_0$. We let $\TL(G)^\diamond$ be the poset obtained by adjoining to $\TL(G)$ a new maximal element which we will denote by $\diamond$. 
\end{definition}

To simplify notation, we will choose a representative for each conjugacy class in $\TL(G)$, and we will identify $\TL(G)$ with the set of chosen representatives. This allows us to take an element $G_0\in\TL(G)$ and work with $G_0$ as an actual subgroup. 

For any twisted Levi subgroup $G_0\subseteq G,$ we get a natural inclusion $\widetilde{\mathcal{B}}(G_0)\subseteq\widetilde{\mathcal{B}}(G)$. This inclusion is compatible with Moy-Prasad subalgebras in the following way. Let $x$ be a point of $\widetilde{\mathcal{B}}(G_0)$. We will abuse notation and refer to its image in $\widetilde{\mathcal{B}}(G)$ again by $x$. Starting from $x$ and a rational number $r$, we can either take the Moy-Prasad subalgebra $\mathfrak{k}_{x,r}\subseteq\mathfrak{g}((t))$ or the Moy-Prasad subalgebra $\mathfrak{k}_{x,r}^{G_0}\subseteq\mathfrak{g}_0((t)).$ Then we have
\[
\mathfrak{k}_{x,r}^{G_0}=\mathfrak{k}_{x,r}\cap\mathfrak{g}_0((t)).
\]

\section{The vector space $\mathfrak{k}_{x,r}/\mathfrak{k}_{x,r+}$}\label{s:stratification}

In this section, we will study, for a fixed $x\in\mathcal{B}(G)$ and $r\in\mathbb{R}$, the vector space $\mathfrak{k}_{x,r}/\mathfrak{k}_{x,r+}$. For $r>0$, this vector space is isomorphic to $K_{x,r}/K_{x,r+}$, and characters of this group play a central role in \cite{MP1}.

The commutation relations for Moy-Prasad subgroups imply that the adjoint action $P_x$ on $\mathfrak{k}_{x,r}/\mathfrak{k}_{x,r+}$ factors through $P_x/P_x^+,$ which is a reductive algebraic group of finite type over $k$. Let $L_x$ denote the quotient $P_x/P_x^+.$ In \cite{MP1}, it was discovered (in the local field context) that semistable $L_x$-orbits in $\mathfrak{k}_{x,r}/\mathfrak{k}_{x,r+}$ control the representation theory of $G((t))$. This section is devoted to the study of such orbits, culminating in Theorem \ref{basecase} (proven in the next section), which reduces the study of these orbits to that of analogous orbits for a twisted Levi subgroup.

It will be helpful to fix $(n,a,\sigma,\Xi)$ as in Lemma \ref{twistedaffinegroup}. We will also fix a maximal $k$-torus in $T_H^{\sigma}$, so we can define the apartment $\mathcal{A}(T)$ in $\mathcal{B}(G)$. However, it is important for applications that our definitions and theorems do not depend on these choices, except when the chosen objects explicitly appear in the statements. This choice-independence is immediate from the structure of the proofs.

\subsection{The semistable locus}

Recall that the semistable locus $(\mathfrak{k}_{x,r}/\mathfrak{k}_{x,r+})^{\circ}$ is defined to be the loci of points which do not contain $0$ in their $L_x$-orbit closure. Its complement is the unstable locus $(\mathfrak{k}_{x,r}/\mathfrak{k}_{x,r+})^{\on{us}}$. The relevance of these loci is explained by the following result, which is an adaptation of Section 6 of \cite{MP1} to our setting.\footnote{See also Theorem 3.2 of our previous paper \cite{yangMP}, which is essentially the same result, but stated in the language of categorical representations.}

\begin{theorem}\label{ssvsdepth}
Let $v$ be an element of $\mathfrak{k}_{x,r}$. Then $v$ lies in $\mathfrak{k}_{y,r+}$ for some $y\in\mathcal{B}(G)$ if and only if the image $\overline{v}$ of $v$ in $\mathfrak{k}_{x,r}/\mathfrak{k}_{x,r+}$ lies in $(\mathfrak{k}_{x,r}/\mathfrak{k}_{x,r+})^{\on{us}}$.
\end{theorem}

The key to proving this is the following lemma.

\begin{lemma}\label{lem:ssvsdepth}
For every point $y$ in $\mathcal{B}(G)$, define $Z_y\subset\mathfrak{k}_{x,r}/\mathfrak{k}_{x,r+}$ to be the image of $\mathfrak{k}_{x,r}\cap\mathfrak{k}_{y,r+}\rightarrow\mathfrak{k}_{x,r}/\mathfrak{k}_{x,r+}$. Then we have
\[
\displaystyle\bigcup_{y\in\mathcal{B}(G)}Z_y=\displaystyle\bigcup_{\substack{y\in\mathcal{B}(G)\\\mathfrak{k}_{x,r+}\subseteq \mathfrak{k}_{y,r+}\subseteq \mathfrak{k}_{x,r}}}Z_y=(\mathfrak{k}_{x,r+}/\mathfrak{k}_{x,r})^{\on{us}}
\]
at the level of $k$-points.
\end{lemma}

\begin{proof}
For convenience we will take $x$ and $y$ to live in the extended building $\widetilde{\mathcal{B}}(G)$ instead, which clearly does not affect the truth of the statement.

First we show that each $Z_y$ is contained in $(\mathfrak{k}_{x,r}/\mathfrak{k}_{x,r+})^{\on{us}}.$ Every two points of $\widetilde{\mathcal{B}}(G)$ lie in a common apartment, and so we can conjugate $x$ and $y$ to lie in $\widetilde{\mathcal{A}}(T)\cong X_*(T)\otimes\mathbb{R}$. Then we have an isomorphism
\[
\mathfrak{k}_{x,r}/\mathfrak{k}_{x,r+}\cong\displaystyle\bigoplus_{\langle\alpha,x\rangle+ i \thinspace=\thinspace r}\mathfrak{h}_{\alpha}^{i}t^i.
\]
The subspace $\mathfrak{h}_{\alpha}^{i}t^i$, for $i$ and $\alpha$ such that $\langle\alpha,x\rangle+ i = r$, lies in $\mathfrak{k}_{y,r+}$ if and only if $\langle\alpha,y\rangle+ i> r,$ or equivalently, $\langle\alpha,y-x\rangle> 0$. It thus follows that
\[
Z_y = \displaystyle\bigoplus_{\substack{\langle\alpha,x\rangle+ i \thinspace=\thinspace r\\ \langle\alpha,y-x\rangle > 0}}\mathfrak{h}_{\alpha}^{i}t^i.
\]

Note that there is a coweight $z\in X_*(T)$ such that $\langle z,\alpha\rangle >0$ for all $\alpha$ with $\langle y-x,\alpha\rangle >0$. Indeed, if $z$ is allowed to lie in $X_*(T)\otimes\mathbb{R}$, the loci of such $z$ is an open cone containing $\alpha$. Every such cone intersects every lattice, including $X_*(T)$, so we can find such a $z$ in $X_*(T)$. By our assumptions, for every point $p\in Z_y,$ we have
\begin{equation}\label{eq:hilbertmumford}
\lim_{a\rightarrow0}z(a)\cdot p=0,
\end{equation}
so the orbit closure of $p$ contains $0$, as desired. 

It remains to show that for every $p\in(\mathfrak{k}_{x,r}/\mathfrak{k}_{x,r+})^{\on{us}}$ there is some $y\in\widetilde{\mathcal{B}}(G)$ with $\mathfrak{k}_{x,r+}\subseteq \mathfrak{k}_{y,r+}\subseteq \mathfrak{k}_{x,r}$ and $p\in Z_y$. We can assume without loss of generality that $x$ lies in $\mathcal{A}(T)$. By the Hilbert-Mumford criterion, there is a 1-parameter subgroup $z:\mathbb{G}_m\rightarrow L_x$ such that (\ref{eq:hilbertmumford}) is satisfied.

As $z$ is a 1-parameter subgroup in the reductive group $L_x$, it lies in a conjugate $\overline{a}T\overline{a}^{-1}$ of $T$, for some $\overline{a}\in L_x$. Then $\overline{a}^{-1}z\overline{a}$ corresponds to some cocharacter $\alpha$ of $T.$ Let $y'\in\mathcal{A}(T)$ be the point $x+\epsilon\cdot\alpha,$ for some sufficiently small $\epsilon > 0.$ Reversing the logic from the first part of our proof, we see that $\overline{a}\cdot p$ lies in $Z_{y'}.$ Let $a$ be any lift of $\overline{a}$ to $P_x$ and take $y=a^{-1}\cdot y'.$ Then we have
\[
p\in\overline{a}^{-1}Z_{y'}=Z_{y}.
\]

It remains to see that $K_{x,r+}\subseteq K_{y,r+}\subseteq K_{x,r}.$ A short computation shows that $K_{x,r+}\subseteq K_{y',r+}\subseteq K_{x,r},$ and conjugating by $a$ gives the desired inclusions.
\end{proof}

\begin{proof}[Proof of Theorem \ref{ssvsdepth}]
    First assume that $v$ lies in $\mathfrak{k}_{y,r+}$ for some $y$. Then $\overline{v}$ must be contained in $Z_y,$ and hence by Lemma \ref{lem:ssvsdepth} lies in the unstable locus.

    Conversely, assume that $\overline{v}$ is in $(\mathfrak{k}_{x,r}/\mathfrak{k}_{x,r+})^{\on{us}}$. By Lemma \ref{lem:ssvsdepth}, there must be some $y\in\mathcal{B}(G)$ with $K_{x,r+}\subseteq K_{y,r+}\subseteq K_{x,r}$ and $\overline{v}\in Z_y$. These conditions imply that $\mathfrak{k}_{y,r+}$ coincides with the preimage of $Z_y$ under the projection $\mathfrak{k}_{x,r}\rightarrow\mathfrak{k}_{x,r}/\mathfrak{k}_{x,r+}.$ It follows that $v$ is in $\mathfrak{k}_{y,r+},$ as desired.
\end{proof}

As a quick corollary, we have the following.

\begin{proposition}\label{rationality}
Assume that $(\mathfrak{k}_{x,r}/\mathfrak{k}_{x,r+})^{\circ}$ is nonempty. Then $r$ is a rational number.
\end{proposition}

\begin{proof}
We can assume without loss of generality that $x$ is in $\mathcal{A}(T)$. By Theorem \ref{ssvsdepth}, it suffices to show that assuming $r$ is irrational, then there is some $y\in\mathcal{B}(G)$ such that $\mathfrak{k}_{x,r}\subseteq\mathfrak{k}_{y,r+}.$ We will in fact produce such a $y$ in $\mathcal{A}(T)$.

Recall that an affine root space $\mathfrak{h}_{\alpha}^it^i$ lies in $\mathfrak{k}_{x,r}$ if and only if $\langle\alpha,x\rangle+i\geq r.$ Consider the set $S$ of pairs $(y,s)$ such that $\mathfrak{k}_{x,r}$ is a subset of $\mathfrak{k}_{y,s}.$ This condition is equivalent to requiring, for every affine root space $\mathfrak{h}_{\alpha}^it^i$ in $\mathfrak{k}_{x,r}$, that we have $\langle\alpha,y\rangle+i\geq s.$ Tautologically, we have $(x,r)\in S$.

A priori, $S$ is the intersection of infinitely many half-spaces in $(X_*(T/T\cap Z_H)\otimes\mathbb{R})\oplus\mathbb{R}$. However, it is clear that the inequality corresponding to an affine root space $\mathfrak{h}_{\alpha}^it^i\subseteq\mathfrak{k}_{x,r}$ implies the inequality corresponding to any affine root space $\mathfrak{h}_{\alpha}^jt^j$ with $j>i.$ It follows that $S$ can in fact be defined by only finitely many half-spaces. Thus the problem of maximizing $s$ over all $(y,s)\in S$ is a linear program with rational coefficients with respect to the lattice $X_*(T/T\cap Z_H)\oplus\mathbb{Z}$. It follows that the maximum possible value of $s$ must be rational (or that $s$ can be arbitrarily large, though this case is easy to rule out) and thus that there exists some $(y,s)\in S$ with $s > r$. But then $\mathfrak{k}_{x,r}\subseteq\mathfrak{k}_{y,s}\subseteq\mathfrak{k}_{y,r+}$, as desired.

\end{proof}

\subsection{The map $q_{x,r}$}

As $G$ is a reductive group over $k((t))$, we can form the GIT quotient $\mathfrak{g}//G$. This gives an affine $k((t))$-variety $C$, which inherits a canonical $\mathbb{G}_m$ action from the scaling action on $\mathfrak{g}$. Using Lemma \ref{twistedaffinegroup}, we can identify the variety $C$:

\begin{lemma}\label{explicitcenter}
Let $\sigma'$ be the automorphism of $\mathfrak{h}//H$ induced by $\sigma$. Then $C$ can be identified with the Galois twist of $(\mathfrak{h}//H)\times_{\Spec k}\Spec k((t))$ by the 1-cocycle which sends $a$ to the automorphism $\sigma'\times a.$

Furthermore, there is a canonical vector space structure on $\mathfrak{h}//H$, and $\sigma'$ is linear with respect to this structure. We can thus define a $k((t))$-linear automorphism $\theta'$ of $(\mathfrak{h}//H)((t^{\frac{1}{n}}))$ by
\[
\theta'(h\cdot t^{\frac{i}{n}})=\zeta^{i}\cdot\sigma'(h)\cdot t^{\frac{i}{n}},
\]
and $C((t))$ can be identified with the $\theta'$-fixed points of $(\mathfrak{h}//H)((t^{\frac{1}{n}}))$. 

All the above identifications are compatible with the aforementioned $\mathbb{G}_m$ actions.
\end{lemma}

\begin{proof}
Most of the desired statements follow from standard Galois descent manipulations, as in Lemma \ref{twistedaffinegroup}. The only statement which is not formal is that $\sigma'$ is linear. To see this, we use the theory of the Kostant section.

For any $\mathfrak{sl}_2$-triple $\{e,h,f\}$ with $e$ a principal nilpotent, Kostant \cite{kostant} proved that the map $e+C_{\mathfrak{h}}(f)\rightarrow\mathfrak{h}//H$ is an isomorphism, where $C_{\mathfrak{h}}(f)$ is the centralizer of $f$ inside $\mathfrak{h}$. This allows us to endow $\mathfrak{h}//H$ with the vector space structure of $C_{\mathfrak{h}}(f)$. As another theorem of Kostant \cite{kostantuniqueness} states that every such $\mathfrak{sl}_2$-triple $\{e,h,f\}$ is conjugate, this vector space structure does not depend on the specific triple and is therefore canonical.

We now construct a specific such $\mathfrak{sl}_2$-triple. We can choose a nonzero element $e_{\alpha}$ in each simple root space of $H$ so that $\sigma$ preserves the collection of these elements. Let $e$ be the sum of the $e_{\alpha}$, and let $h$ be the sum of positive co-roots in $H$. Then there is a unique element $f$ such that $\{e,h,f\}$ is a $\mathfrak{sl}_2$-triple.

As $\sigma$ preserves the collection of the $e_{\alpha}$, as well as $T_H$ and $B_H$, it must preserve the elements $e$ and $h$ (and hence $f$). It clearly acts by a linear automorphism on $C_{\mathfrak{g}}(f)$, so it does on $\mathfrak{h}//H$ as well, as desired.
\end{proof}

For future reference, we note the following simple consequence of the above proof.

\begin{lemma}\label{quasisplitkostant}
There is a $\mathfrak{sl}_2$-triple $\{e_{\mathfrak{g}}, h_{\mathfrak{g}},f_{\mathfrak{g}}\}$ in $\mathfrak{g}$ such that the map $e_{\mathfrak{g}}+C_{\mathfrak{g}}(f_{\mathfrak{g}})\rightarrow C$ is an isomorphism.
\end{lemma}

\begin{proof}
Let $e,h,$ and $f$ be as in the proof of Lemma \ref{explicitcenter}. As $\sigma$ preserves $e$, we can use Lemma \ref{twistedaffine} to descend the image of $e$ in $\mathfrak{h}((t))$ to an element $e_{\mathfrak{g}}$ of $\mathfrak{g}.$ We can do the same for $h$ and $f$. The lemma now follows by descending Kostant's theorem.
\end{proof}

The identifications of Lemma \ref{explicitcenter} are not canonical, as they depend on the choices made in Lemma \ref{twistedaffinegroup}. However, one does obtain some canonical structures on $C$.

\begin{lemma}\label{canonicalstructures}
The vector space structure on $C$ produced by Lemma \ref{explicitcenter} does not depend on the choice of $(n,a,\sigma,\Xi)$.
Furthermore, the $\mathbb{Q}$-indexed grading on $C((t))$ which corresponds under Lemma \ref{explicitcenter} to the degree grading on $(\mathfrak{h}//H)((t))$ is also independent of the above choices.
\end{lemma}

\begin{proof}
Assume that we have two different choices $(n,a,\sigma,\Xi)$ and $(n',a',\sigma',\Xi')$ as in Lemma \ref{twistedaffinegroup}. Passing to a common multiple, we can assume that $(n,a)=(n',a').$ Then $(\sigma,\Xi)$ and $(\sigma',\Xi')$ must differ by an automorphism of $H\times_{\Spec k}\Spec k((t^{\frac{1}{n}})).$ After potentially replacing $k((t{\frac{1}{n}}))$ by an extension, such an automorphism can be written as the composition of an inner automorphism and a pinning-preserving automorphism. An inner automorphism acts trivially on $(\mathfrak{h}//H)\times_{\Spec k}\Spec k((t^{\frac{1}{n}}))$, so we are reduced to the case of a pinning-preserving automorphism. Such an automorphism must be the base change to $\Spec k((t^{\frac{1}{n}}))$ of a pinning-preserving $k$-automorphism of $H$, which, as we saw in the proof of Lemma \ref{explicitcenter}, induces a linear automorphism of $\mathfrak{h}//H$. This in turn induces a linear automorphism of $(\mathfrak{h}//H)\times_{\Spec k}\Spec k((t^{\frac{1}{n}}))$ and a degree-preserving linear automorphism of $(\mathfrak{h}//H)((t))$, as desired.
\end{proof}

For $j$ a rational number, let $C((t))^j$ denote the $j$th graded component of $C((t))$ under the degree grading. There is also another natural grading on $C((t))$. Recall the canonical $\mathbb{G}_m$ action on $C$. We claim that this action is linear. Indeed, this follows from Lemma \ref{explicitcenter}, together with the fact that the canonical $\mathbb{G}_m$ action on $\mathfrak{h}//H$ is linear. We thus get a decomposition
\[
C=\displaystyle\bigoplus_{i\in\mathbb{Z}}C_i
\]
of $C$ into its weight $i$ components. As the $\mathbb{G}_m$ action is expanding, $C_i$ can only be nonzero when $i$ is positive. We note that the induced grading on $C((t))$ commutes with the degree grading, so we have
\[
C((t))=\displaystyle\bigoplus_{\substack{i\in\mathbb{Z}_+\\j\in\mathbb{Q}}}C_i((t))^j.
\]

The $\mathbb{G}_m$ action on $\mathfrak{h}//H$ induces a similar decomposition
\[
\mathfrak{h}//H=\displaystyle\bigoplus_{i\in\mathbb{Z}_+}(\mathfrak{h}//H)_i,
\]
and each $C_i$ is the appropriate Galois twist of $(\mathfrak{h}//H)_i.$ It follows that each $C_i((t))$ can be identified with the $\theta'$-fixed points of $(h//H)_i((t^{\frac{1}{n}})).$

By definition, there is a natural map $\mathfrak{g}\rightarrow C$ of $k((t))$-varieties, which induces a map $q:\mathfrak{g}((t))\rightarrow C((t))$ of ind-schemes. This map has the following nice compatibility with the Moy-Prasad filtration.

\begin{theorem}\label{depthcenter}
Choose $x\in\mathcal{B}(G)$ and a real number $r$.
\begin{enumerate}
\item The map $q$ sends $\mathfrak{k}_{x,r}$ into
\[
\displaystyle\bigoplus_{\substack{i\in\mathbb{Z}_+,j\in\mathbb{Q}\\ j\thinspace\geq\thinspace r\cdot i}}C_i((t))^j.
\]
\item The composition
\[
\mathfrak{k}_{x,r}\rightarrow\displaystyle\bigoplus_{\substack{i\in\mathbb{Z}_+,j\in\mathbb{Q}\\ j\thinspace\geq\thinspace r\cdot i}}C_i((t))^j\rightarrow\displaystyle\bigoplus_{i\in\mathbb{Z}_+}C_i((t))^{r\cdot i}
\]
factors through $\mathfrak{k}_{x,r}/\mathfrak{k}_{x,r+}.$
\end{enumerate}
\end{theorem}

\begin{proof}
We will prove the theorem by reducing to the split case, using the commutative diagram
\begin{center}
\begin{tikzcd}
\mathfrak{g}((t))\arrow[r, "q"]\arrow[d] & C((t))\arrow[d]\\
\mathfrak{g}((t^{\frac{1}{n}}))\arrow[r] \arrow[d,"\cong"] & C((t^{\frac{1}{n}})) \arrow[d,"\cong"]\\
\mathfrak{h}((t^{\frac{1}{n}}))\arrow[r] & (\mathfrak{h}//H)((t^{\frac{1}{n}})).
\end{tikzcd}
\end{center}
As $\mathfrak{k}_{x,r}$ and $\mathfrak{k}_{gx,r}$ are conjugate for any $g\in G(k((t)))$, their images under $q$ agree. We can thus assume without loss of generality that $x\in\mathcal{A}_*(T)$, where $T$ is as in Section \ref{ss:building}. Set $y\in\mathcal{A}_*(T_H)$ to be the image of $x$ under the inclusion of $\mathcal{A}_*(T)$ in $\mathcal{A}_*(T_H)$.

Let $\mathfrak{k}_{y,r}^H$ denote the subspace
\[
\displaystyle\bigoplus_{\substack{\alpha\in X^*(T_H)\\ i\in\frac{1}{n}\mathbb{Z} \\ \langle\alpha,y\rangle+i\thinspace\geq\thinspace r}}\mathfrak{h}_{\alpha}t^i
\]
of $\mathfrak{h}((t^{\frac{1}{n}})).$ (We include the $\alpha\in X^*(T_H)$ subscript here because it differs from the $\alpha\in X^*(T)$ subscript for subspaces of $g((t))$.) It follows immediately from the definitions that the inclusion $\mathfrak{g}((t))\rightarrow\mathfrak{h}((t^{\frac{1}{n}}))$ sends $\mathfrak{k}_{x,r}$ into $\mathfrak{k}^H_{y,r}$. From the aforementioned commutative diagram, we see that to show part 1), it suffices to show that the arrow
\[
\mathfrak{h}((t^{\frac{1}{n}}))\rightarrow(\mathfrak{h}//H)((t^{\frac{1}{n}}))
\]
sends $\mathfrak{k}^H_{y,r}$ into
\begin{equation}\label{eq:splittarget}
\displaystyle\bigoplus_{\substack{i\in\mathbb{Z}_+,j\in\frac{1}{n}\mathbb{Z}\\ j\thinspace\geq\thinspace r\cdot i}}(\mathfrak{h}//H)_it^j.
\end{equation}

We can similarly reduce part 2) to the split case. Indeed, let $\mathfrak{k}_{y,r+}^H$ denote the subspace
\[
\displaystyle\bigoplus_{\substack{\alpha\in X^*(T_H)\\ i\in\frac{1}{n}\mathbb{Z} \\ \langle\alpha,y\rangle+i\thinspace>\thinspace r}}\mathfrak{h}_{\alpha}t^i.
\]
Note that the preimage of $\mathfrak{k}_{y,r+}^H$ along $\mathfrak{k}_{x,r}\rightarrow\mathfrak{k}_{y,r}^H$ equals $\mathfrak{k}_{x,r+}.$ Thus, part 2) of the theorem would follow if we knew that the map
\begin{equation}\label{eq:splittargetcomposition}
\mathfrak{k}^H_{y,r}\rightarrow\displaystyle\bigoplus_{\substack{i\in\mathbb{Z}_+,j\in\frac{1}{n}\mathbb{Z}\\ j\thinspace\geq\thinspace r\cdot i}}(\mathfrak{h}//H)_it^j\rightarrow\displaystyle\bigoplus_{\substack{i\in\mathbb{Z}_+,j\in\frac{1}{n}\mathbb{Z}\\ j\thinspace=\thinspace r\cdot i}}(\mathfrak{h}//H)_it^j
\end{equation}
factored through $\mathfrak{k}^H_{y,r}/\mathfrak{k}^H_{y,r+}.$

These two statements were effectively proven in Sections 4.4.2 and 4.4.3 of \cite{yangMP}, but in a noncommutative context. We will repeat the proof here for completeness.

The ind-scheme $\mathfrak{h}((t^{\frac{1}{n}}))$ is ind-affine, and thus naturally corresponds to a commutative topological algebra (see Section 19.2 of \cite{FG2}). Explicitly, if $\mathfrak{h}((t^{\frac{1}{n}}))$ is written as the direct limit $\varinjlim X_i$ of affine schemes under closed embeddings, then the corresponding commutative topological algebra would be $\varprojlim\mathcal{O}(X_i).$ In the case of $\mathfrak{h}((t^{\frac{1}{n}}))$, this construction can be described explicitly as follows. The topological vector space $\mathfrak{h}((t^{\frac{1}{n}}))$ is Tate, so there is a canonical topology on its dual $\mathfrak{h}((t^{\frac{1}{n}}))^*$. Then the commutative topological algebra associated to the ind-scheme $\mathfrak{h}((t^{\frac{1}{n}}))$ is the completion $\overline{\Sym}(\mathfrak{h}((t^{\frac{1}{n}}))^*)$ of $\Sym(\mathfrak{h}((t^{\frac{1}{n}}))^*)$ with respect to the topology induced from that of $\mathfrak{h}((t^{\frac{1}{n}}))^*$. Similarly, to the ind-scheme $(\mathfrak{h}//H)((t^{\frac{1}{n}})$ this construction associates the algebra $\overline{\Sym}((\mathfrak{h}//H)((t^{\frac{1}{n}}))^*)$.

The map $\mathfrak{h}((t^{\frac{1}{n}}))\rightarrow(\mathfrak{h}//H)((t^{\frac{1}{n}}))$ then corresponds to a map of topological algebras $\overline{\Sym}((\mathfrak{h}//H)((t^{\frac{1}{n}}))^*)\rightarrow\overline{\Sym}(\mathfrak{h}((t^{\frac{1}{n}}))^*)$. We will need an explicit description of this map. Choose a basis $a_{i,j}$ for each vector space $(\mathfrak{h}//H)_i^*$. Then for every $m\in\frac{1}{n}\mathbb{Z}$, we have an element $a_{i,j}^m\in(\mathfrak{h}//H)_i((t))^*$ defined by
\[
\langle a_{i,j}^m,ct^s\rangle=\langle a_{i,j},c\rangle\delta_{m,s}.
\]
The topological vector space $(\mathfrak{h}//H)((t^{\frac{1}{n}}))^*$ can be described as the inverse limit of the discrete vector spaces $V_o$ spanned by the $a_{i,j}^m$ with $m\geq o$, as $o$ ranges over all elements of $\frac{1}{n}\mathbb{Z}.$

We similarly choose basises $b_{\alpha,j}$ of each dual root space $\mathfrak{h}_{\alpha}^*$ and define elements $b_{\alpha,j}^m\in\mathfrak{h}((t^{\frac{1}{n}}))^*$ analogously to the $a_{i,j}^m$. Assume that the pullback of $a_{i,j}$ along $\mathfrak{h}\rightarrow\mathfrak{h}//H$ is equal to
\[
\displaystyle\sum_{l}c_{l}b_{\alpha_{l,1},j_{l,1}}b_{\alpha_{l,2},j_{l,2}}\cdots b_{\alpha_{l,i},j_{l,i}},
\]
with the $c_{l}$ elements of $k$. (We are using that the compatibility of $\mathfrak{h}\rightarrow\mathfrak{h}//H$ with the natural $\mathbb{G}_m$ actions tells us that the pullback of $a_{i,j}$ is a homogeneous polynomial of degree $i$ in the $b_{\alpha,j'}$.) Then a short computation with Laurent series shows that the image of $a_{i,j}^m$ in $\overline{\Sym}(\mathfrak{h}((t^{\frac{1}{n}}))^*)$ is equal to
\[
\displaystyle\sum_{l}c_l\sum_{\substack{m_1,m_2,\cdots, m_i\in\frac{1}{n}\mathbb{Z}\\ m_1+m_2+\cdots+m_i=m}}b_{\alpha_{l,1},j_{l,1}}^{m_1}b_{\alpha_{l,2},j_{l,2}}^{m_2}\cdots b_{\alpha_{l,i},j_{l,i}}^{m_i}.
\]

The subscheme (\ref{eq:splittarget}) of $(\mathfrak{h}//H)((t^{\frac{1}{n}}))$ is cut out by the vanishing of the $a_{i,j}^m$ with $m<r\cdot i$. So for part 1), it suffices to show that image of each such $a_{i,j}^m$ in $\overline{\Sym}(\mathfrak{h}((t^{\frac{1}{n}}))^*)$ vanishes on $\mathfrak{k}_{x,r}$. Fix such $i,j,$ and $m$, and assume as before that the pullback of $a_{i,j}$ along $\mathfrak{h}\rightarrow\mathfrak{h}//H$ is equal to
\[
\displaystyle\sum_{l}c_{l}b_{\alpha_{l,1},j_{l,1}}b_{\alpha_{l,2},j_{l,2}}\cdots b_{\alpha_{l,i},j_{l,i}}.
\]
This polynomial must be $H$-invariant, hence $T_H$-invariant, so for any $l$, we must have
\[
\alpha_{l,1}+\alpha_{l,2}+\cdots+\alpha_{l,i}=0.
\]

It follows that for any term $b_{\alpha_{l,1},j_{l,1}}^{m_1}b_{\alpha_{l,2},j_{l,2}}^{m_2}\cdots b_{\alpha_{l,i},j_{l,i}}^{m_i}$ of the image of $a_{i,j}^m$ in $\overline{\Sym}(\mathfrak{h}((t^{\frac{1}{n}}))^*)$, we have
\[
\displaystyle\sum_{z=1}^i(\langle\alpha_{l,z},y\rangle+r-m_z)=\langle 0,y\rangle+r\cdot i-m > 0.
\]
There is thus some value of $z$ for which $\langle\alpha_{l,z},y\rangle+r > m_z.$ This implies that $b_{\alpha_{l,z},j_{l,z}}^{m_z}$ acts trivially on $\mathfrak{k}_{y,r}^H$. Therefore each term of the image of $a_{i,j}^m$ vanishes on $\mathfrak{k}_{y,r}^H$, as desired.

To prove part 2), we examine the above proof a little more closely. The composition (\ref{eq:splittargetcomposition}) is determined by the pullbacks of the $a_{i,j}^m$ with $m=r\cdot i$ to $\mathfrak{k}_{y,r}^H$. Fix such $i,j,$ and $m$. 

Following the lines of the proof of part 1), we see that a term $b_{\alpha_{l,1},j_{l,1}}^{m_1}b_{\alpha_{l,2},j_{l,2}}^{m_2}\cdots b_{\alpha_{l,i},j_{l,i}}^{m_i}$ of the image of $a_{i,j}^m$ in $\Sym(\mathfrak{h}((t^{\frac{1}{n}}))^*)$ can only be nontrivial on $\mathfrak{k}_{y,r}^H$ if $\langle\alpha_{l,z},y\rangle+r-m_z\leq 0$ for all $z$. But we have
\[
\displaystyle\sum_{z=1}^i(\langle\alpha_{l,z},y\rangle+r-m_z)=\langle 0,y\rangle+r\cdot i-m = 0,
\]
so we must have $\langle\alpha_{l,z},y\rangle+r-m_z= 0$ for all $z$. But this implies that our term is pulled back from $\mathfrak{k}_{y,r}^H/\mathfrak{k}_{y,r+}^H$, as desired.
\end{proof}

Motivated by the above theorem, we define $C((t))_{\geq r}, C((t))_{>r}$ and $C((t))_{=r}$ to be 
\begin{align*}
\displaystyle\bigoplus_{\substack{i\in\mathbb{Z}_+,j\in\mathbb{Q}\\ j\thinspace\geq\thinspace r\cdot i}}&C_i((t))^j,\\
\displaystyle\bigoplus_{\substack{i\in\mathbb{Z}_+,j\in\mathbb{Q}\\ j\thinspace>\thinspace r\cdot i}}&C_i((t))^j,
\end{align*}
and
\[
\displaystyle\bigoplus_{\substack{i\in\mathbb{Z}_+,j\in\mathbb{Q}\\ j\thinspace=\thinspace r\cdot i}}C_i((t))^j,
\]
respectively. Then Theorem \ref{depthcenter} gives us a map $\mathfrak{k}_{x,r}/\mathfrak{k}_{x,r+}\rightarrow C((t))_{=r}$ which we will call $q_{x,r}$. As the map $\mathfrak{k}_{x,r}\rightarrow C((t))$ is $P_x$-invariant by definition, $q_{x,r}$ is $L_x$-invariant. It is also naturally compatible with extensions of the field $k((t))$. In particular, if $y\in\mathcal{B}(G,k((t^{\frac{1}{n}})))$ denotes the image of $x$ along the inclusion $\mathcal{B}(G)\rightarrow\mathcal{B}(G,k((t^{\frac{1}{n}})))$, then there is a commutative diagram
\begin{center}
\begin{tikzcd}
\mathfrak{k}_{x,r}/\mathfrak{k}_{x,r+}\arrow[r, "q_{x,r}"]\arrow[d] & C((t))_{=r} \arrow[d]\\
\mathfrak{k}_{y,r}^H/\mathfrak{k}_{y,r+}^H\arrow[r] & \displaystyle\bigoplus_{\substack{i\in\mathbb{Z}_+,j\in\frac{1}{n}\mathbb{Z}\\ j\thinspace=\thinspace r\cdot i}}(\mathfrak{h}//H)_it^j.
\end{tikzcd}
\end{center}

Call the bottom arrow $q_{y,r}^H.$ We can squeeze a little more juice from the methods from the proof of Theorem \ref{depthcenter} and describe $q_{y,r}^H$.

\begin{lemma}\label{qxrdescription}
The diagram
\begin{center}
\begin{tikzcd}
\mathfrak{k}_{y,r}^H/\mathfrak{k}_{y,r+}^H\arrow[r, "q_{y,r}^H"] \arrow[d] & \displaystyle\bigoplus_{\substack{i\in\mathbb{Z}_+,j\in\frac{1}{n}\mathbb{Z}\\ j\thinspace=\thinspace r\cdot i}}(\mathfrak{h}//H)_it^j\arrow[d]\\
\mathfrak{h} \arrow[r] & \mathfrak{h}//H
\end{tikzcd}
\end{center}
where the left arrow is the composition
\[
\mathfrak{k}_{y,r}^H/\mathfrak{k}_{y,r+}^H\cong\displaystyle\bigoplus_{\substack{\alpha\in X^*(T_H)\\ i\in\frac{1}{n}\mathbb{Z} \\ \langle\alpha,y\rangle+i\thinspace=\thinspace r}}\mathfrak{h}_{\alpha}t^i \xrightarrow{t\mapsto 1}\displaystyle\bigoplus_{\substack{\alpha\in X^*(T_H)\\ r-\langle\alpha,y\rangle\in\frac{1}{n}\mathbb{Z}}}\mathfrak{h}_{\alpha}\rightarrow\mathfrak{h}
\]
and the right arrow is the composition
\[
\displaystyle\bigoplus_{\substack{i\in\mathbb{Z}_+,j\in\frac{1}{n}\mathbb{Z}\\ j\thinspace=\thinspace r\cdot i}}(\mathfrak{h}//H)_it^j\cong\displaystyle\bigoplus_{\substack{i\in\mathbb{Z}_+,r\cdot i\in\frac{1}{n}\mathbb{Z}}}(\mathfrak{h}//H)_it^{r\cdot i}\xrightarrow{t\mapsto 1}\displaystyle\bigoplus_{\substack{i\in\mathbb{Z}_+,r\cdot i\in\frac{1}{n}\mathbb{Z}}}(\mathfrak{h}//H)_i\rightarrow\mathfrak{h}//H
\]
is commutative.
\end{lemma}

\begin{proof}

We choose $a_{i,j}$ and $b_{\alpha,j}$ as in the proof of Theorem \ref{depthcenter}. As the $a_{i,j}$ form a basis for $(\mathfrak{h}//H)^*$, it suffices to show that for any $i$ and $j$, the two pullbacks of $a_{i,j}$ to $\mathfrak{k}_{y,r}^H/\mathfrak{k}_{y,r+}^H$ along the two paths in the diagram agree. Again, say that the pullback of $a_{i,j}$ along $\mathfrak{h}\rightarrow\mathfrak{h}//H$ is equal to
\[
\displaystyle\sum_{l}c_{l}b_{\alpha_{l,1},j_{l,1}}b_{\alpha_{l,2},j_{l,2}}\cdots b_{\alpha_{l,i},j_{l,i}}.
\]
Then when we further pullback this expression to $\mathfrak{k}_{y,r}^H/\mathfrak{k}_{y,r+}^H$, we get
\[
\displaystyle\sum_{\substack{l\\ \forall z, r-\langle\alpha_{l,z},y\rangle\in\frac{1}{n}\mathbb{Z}}}c_{l}b_{\alpha_{l,1},j_{l,1}}^{r-\langle\alpha_{l,1},y\rangle}b_{\alpha_{l,2},j_{l,2}}^{r-\langle\alpha_{l,2},y\rangle}\cdots b_{\alpha_{l,i},j_{l,i}}^{r-\langle\alpha_{l,i},y\rangle}.
\]
Because the sum of the $\alpha_{l,z}$ is equal to zero for any $l$, we see that for this pullback to be nonzero we must have
\[
r\cdot i=\displaystyle\sum_{z=1}^i(r-\langle\alpha_{l,z},y\rangle)\in\frac{1}{n}\mathbb{Z}.
\]

On the other hand, the pullback of $a_{i,j}$ to
\[
\displaystyle\bigoplus_{\substack{p\in\mathbb{Z}_+,q\in\frac{1}{n}\mathbb{Z}\\ q\thinspace=\thinspace r\cdot p}}(\mathfrak{h}//H)_pt^q
\]
vanishes when $r\cdot i$ is not in $\frac{1}{n}\mathbb{Z}$ and equals $a_{i,j}^{r\cdot i}$ otherwise. If $r\cdot i\not\in\frac{1}{n}\mathbb{Z}$, then both pullbacks vanish and we are done, so assume that $r\cdot i\in\frac{1}{n}\mathbb{Z}.$ The proof of Theorem \ref{depthcenter} shows that the pullback of $a_{i,j}^{r\cdot i}$ to $\mathfrak{k}_{y,r}/\mathfrak{k}_{y,r+}$ can be expressed as the sum of the terms $b_{\alpha_{l,1},j_{l,1}}^{m_1}b_{\alpha_{l,2},j_{l,2}}^{m_2}\cdots b_{\alpha_{l,i},j_{l,i}}^{m_i}$ where for each $z$, we have $m_z=r-\langle\alpha_{l,z},y\rangle.$ But this evidently agrees with the other pullback, so we are done.
\end{proof}

\subsection{GIT of $\mathfrak{k}_{x,r}/\mathfrak{k}_{x,r+}$}\label{ss:grading}

It turns out that the action of $L_x$ on $\mathfrak{k}_{x,r}/\mathfrak{k}_{x,r+}$ must be of a rather restricted form. To see this, we will use the theory of graded Lie algebras. For this subsection, we will assume that $x$ lies in $\mathcal{A}(T)$ unless otherwise stated.

The values of $x$ and $\sigma$ determine a $\mathbb{R}/\mathbb{Z}$-grading on $\mathfrak{h}$ by the formula
\[
\mathfrak{h}_j=\displaystyle\bigoplus_{\substack{\alpha\in X^*(T)\\i\in\mathbb{R}/\mathbb{Z}\\ \langle\alpha,x\rangle+i\cong j\pmod{1}}}\mathfrak{h}_{\alpha}^i.
\]
It is quickly checked that a choice of identification as in Lemma \ref{twistedaffine} determines an isomorphism between $\mathfrak{k}_{x,r}/\mathfrak{k}_{x,r+}$ and the $r$-th graded component $\mathfrak{h}_r.$ In particular, for $r=0$, we see that the Lie algebra of $L_x$ is identified with $\mathfrak{h}_0.$ Thus the action of $L_x$ on $\mathfrak{k}_{x,r}/
\mathfrak{k}_{x,r+}$ is essentially the setting studied by Vinberg \cite{Vinberg} in characteristic $0$ and by Levy \cite{Levy} in positive characteristic. We need the following strengthening of Proposition \ref{rationality} to deal with the cases not covered by Vinberg and Levy.

\begin{lemma}\label{baddenominator}
Assume that there is no exponent $e$ of $\mathfrak{h}$ such that $e\cdot r\in\frac{1}{n}\mathbb{Z}$. Then every element of $\mathfrak{h}_r$ is both $L_x$-unstable and nilpotent.
\end{lemma}

\begin{proof}
First we prove that every element of $\mathfrak{h}_r$ is nilpotent (as an element of $\mathfrak{h}.$) By Lemma \ref{qxrdescription}, it suffices to show that the map $q^H_{y,r}$ is identically zero. Recall that the target of $q^H_{y,r}$ is
\[
\displaystyle\bigoplus_{\substack{i\in\mathbb{Z}_+,j\in\frac{1}{n}\mathbb{Z}\\ j\thinspace=\thinspace r\cdot i}}(\mathfrak{h}//H)_it^j.
\]
As $(\mathfrak{h}//H)_i$ is nonempty only for $i$ an exponent of $\mathfrak{h}$, our assumption ensures that the target consists only of the zero point, as desired.

It remains to show that every nilpotent element $z\in\mathfrak{h}_r$ is $L_x$-unstable. As there are finitely many nilpotent orbits in $\mathfrak{h}$, there must be some nilpotent orbit $o\subseteq\mathfrak{h}$ whose intersection with $\mathfrak{h}_r$ is dense. Then for any point $p\in o\cap\mathfrak{h}_r$, the tangent space $[\mathfrak{h},p]\cap\mathfrak{h}_r$ of $o\cap\mathfrak{h}_r$ at $p$ must equal $\mathfrak{h}_r$. By the properties of the grading on $\mathfrak{h}$, we have $[\mathfrak{h},p]\cap\mathfrak{h}_r=[\mathfrak{h}_0,p]$, so it follows that we must have $[\mathfrak{h}_0,p]=\mathfrak{h}_r$ for any $p\in o\cap\mathfrak{h}_r.$ But $[\mathfrak{h}_0,p]$ is the tangent space to $L_x\cdot p$ at $p$, so $o\cap\mathfrak{h}_r$ must consist of a single $L_x$-orbit. This implies that every element of $\mathfrak{h}_r$ is $L_x$-unstable, as desired. 
 
\end{proof}

Let $\mathfrak{c}$ be a Cartan subspace of $\mathfrak{h}_r$, i.e., a maximal subspace consisting of commuting semisimple elements. The Weyl group $W_{\mathfrak{c}}$ of $\mathfrak{c}$ is the quotient $N_{L_0}(\mathfrak{c})/Z_{L_0}(\mathfrak{c}).$ Let us translate some of the results of \cite{Vinberg} and \cite{Levy} to our setting.

\begin{proposition}\label{vinbergsummary}
\begin{enumerate}
\item An element $z\in\mathfrak{h}_r$ is $L_x$-unstable iff $z$ is nilpotent.
\item An element $z\in\mathfrak{h}_r$ has a closed $L_x$-orbit iff $z$ is semisimple.
\item The group $W_{\mathfrak{c}}$ is finite.
\item The natural map of GIT quotients $\mathfrak{c}/W_{\mathfrak{c}}\rightarrow\mathfrak{h}_r//L_x$ is an isomorphism.
\item Two elements of $\mathfrak{h}_r$ have the same image in $\mathfrak{h}_r//L_x$ iff their semisimple parts lie in the same closed $L_x$-orbit.
\end{enumerate}
\end{proposition}
\begin{proof}

If there is no exponent $e$ of $\mathfrak{h}$ such that $e\cdot r\in\frac{1}{n}\mathbb{Z}$, then all assertions of the proposition are tautological. So we may assume that $e\cdot r\in\frac{1}{n}\mathbb{Z}$ for some exponent $e$. In particular, $r$ is rational and its denominator (in lowest terms) is not a multiple of $\on{char} k$ (by our assumptions on $G$ and $\on{char} k$). After possibly modifying our choice of $x$ (but without changing $\mathfrak{h}_r$), we may further assume that $x\in X_*(T/(T\cap Z_H))\otimes\mathbb{Q}$.

Let $d$ be the smallest multiple of $n$ such that $dr\in\mathbb{Z}$ and let $\mathfrak{h}'$ denote the Lie algebra $\mathfrak{h}_0\oplus\mathfrak{h}_r\oplus\cdots\oplus\mathfrak{h}_{(d-1)r}.$ It admits a natural $\mathbb{Z}/d\mathbb{Z}$-grading defined by $\mathfrak{h}'_{i}=\mathfrak{h}_{i\cdot r}.$ Arguing as in the proof of Theorem 2.2.9(i) of Kawanaka\cite{kawanaka}, one constructs a connected reductive subgroup $H'\subseteq H$ which induces the inclusion of Lie algebras $\mathfrak{h}'\subseteq\mathfrak{h}.$ By construction, $H'$ contains the identity component of $T^{\sigma}$ and is $\sigma$-invariant. 

Let $\zeta_d\in k^*$ be such that $\zeta_d^{d/n}=\zeta$, and choose a homomorphism $\mu:\mathbb{Q}/\mathbb{Z}\rightarrow k^*$ which sends $\frac{1}{d}$ to $\zeta_d.$ If $x=\frac{y}{d'}$ for $y\in X_*(T/(T\cap Z_H))$ and $d'\in\mathbb{Z}$, then we define $x_{\mu}\in T/T\cap Z_H$ to be the element $y(\mu(\frac{1}{d'}))$, which is independent of the choice of $d'$ and $y$. A short computation shows that $\on{ad}\mu_x\circ\sigma$ acts by multiplication with $\zeta_d^i$ on $\mathfrak{h}'_i.$ As $\on{ad}\mu_x\circ\sigma:\mathfrak{h}'\rightarrow\mathfrak{h}'$ is induced by the automorphism $\on{Ad}\mu_x\circ\sigma$ of $H$, we are in the setting of \cite{Vinberg} (in characteristic zero\footnote{Technically, \cite{Vinberg} only works over $\mathbb{C}$, instead of a general algebraically closed field of characteristic $0$. However, the proofs of the results we use from loc.cit. do not use anything specific to $\mathbb{C}$. Alternatively, the standard Lefschetz principle argument can be used to reduce this proposition to the case of $k=\mathbb{C}$.}) and \cite{Levy} (in positive characteristic).

In characteristic zero, the parts of the proposition now follow from, respectively, Proposition 1, Proposition 3, Proposition 8, Theorem 7, and Theorem 3 of \cite{Vinberg}. In positive characteristic, parts 1,2, and 5 of the proposition follow from Corollary 2.13 of \cite{Levy}, while part 3 follows from the same argument as in the proof of Proposition 8 of \cite{Vinberg}. Finally, part 4 is Theorem 2.18 of \cite{Levy}.
\end{proof}
To use this proposition, it will be helpful to translate the notion of semisimplicity of an element of $\mathfrak{h}_r$ back to a notion for $\mathfrak{k}_{x,r}/\mathfrak{k}_{x,r+}.$ For this purpose, we can identify $\mathfrak{k}_{x,r}/\mathfrak{k}_{x,r+}$ with the subspace
\[
\displaystyle\bigoplus_{\langle\alpha,x\rangle+i=r}\mathfrak{h}_{\alpha}^it^i
\]
of $\mathfrak{g}((t)).$ This gives us a map $\mathfrak{h}_r\cong\mathfrak{k}_{x,r}/\mathfrak{k}_{x,r+}\rightarrow\mathfrak{g}((t))$ that we will denote by $f_{x,r}.$ These maps are naturally compatible with the Lie bracket, i.e., for any $r_1, r_2\in\mathbb{R}$, there is a commutative diagram

\begin{center}
\begin{tikzcd}
\mathfrak{h}_{r_1}\times\mathfrak{h}_{r_2} \arrow[r, "{[-,-]}"]\arrow[d, "f_{x,r_1}\times f_{x,r_2}"] & \mathfrak{h}_{r_1+r_2}\arrow[d, "f_{x,r_1+r_2}"]\\
\mathfrak{g}((t))\times\mathfrak{g}((t))\arrow[r, "{[-,-]}"] & \mathfrak{g}((t)).
\end{tikzcd}
\end{center}

\begin{lemma}\label{semisimpleagreement}
    An element $z$ of $\mathfrak{h}_r$ is semisimple (resp. nilpotent) in $\mathfrak{h}$ if and only if $f_{x,r}(z)$ is semisimple (resp. nilpotent) in $\mathfrak{g}((t))$. More generally, if the Jordan decomposition of $z$ is $z=z_s+z_n$, then the Jordan decomposition of $f_{x,r}(z)$ is $f_{x,r}(z)=f_{x,r}(z_s)+f_{x,r}(z_n).$
\end{lemma}
\begin{proof}

    By Proposition \ref{baddenominator}, we can assume that $r$ is rational with denominator not a multiple of $\on{char} k$. The notions of semisimplicity, nilpotence, and Jordan decomposition are invariant under extension of base field. We can thus replace $k((t))$ with some field $k((t^\frac{1}{m}))$ containing $t^r$. Relabeling $t$ to be $t^{\frac{1}{m}}$, we can assume that $m=1$, so $r$ is an integer.
    
    The element $z$ is semisimple iff $\on{ad} z$ is a semisimple endomorphism of $\mathfrak{h}$. As $r$ is an integer, $\on{ad} z$ sends each $\mathfrak{h}_i$ to itself, and $\on{ad} z$ is semisimple iff each restriction $\on{ad} z|_{\mathfrak{h}^i}$ is semisimple.
    
    As a vector space, $\mathfrak{g}((t))$ becomes isomorphic to
    \begin{equation}\label{eq:gradedaffine}
    \displaystyle\bigoplus_{i\in\mathbb{R}/\mathbb{Z}}\mathfrak{h}_i((t))
    \end{equation}
    after choosing a lift of every $i\in\mathbb{R}/\mathbb{Z}$ to $\mathbb{R}.$ The endomorphism $\on{ad} f_{x,r}(z)$ acts on each subspace $\mathfrak{h}_i((t))$ by $t^r\cdot(\on{ad} z|_{\mathfrak{h}_i}\otimes_k k((t))).$ It is clear that $\on{ad} f_{x,r}(z)$ is semisimple iff each $\on{ad} z|_{\mathfrak{h}_i}$ is semisimple, as desired.
    
    Nilpotence of $z$ is similarly equivalent to the combination of $\on{ad} z$ being a nilpotent endomorphism and $z$ lying in $[\mathfrak{h},\mathfrak{h}].$ The same proof as with semisimplicity shows that $\on{ad} z$ is nilpotent iff $\on{ad} f_{x,r}(z)$ is nilpotent. Under the isomorphism of (\ref{eq:gradedaffine}), $[\mathfrak{g},\mathfrak{g}]((t))$ corresponds to the subspace
    \[
    \displaystyle\bigoplus_{i,j\in\mathbb{R}}[\mathfrak{h}_i,\mathfrak{h}_j]((t)).
    \]
    It follows that $z$ is in $[\mathfrak{h},\mathfrak{h}]$ iff $f_{x,r}(z)$ is in $[\mathfrak{g},\mathfrak{g}]((t)),$ so $z$ is nilpotent iff $f_{x,r}(z)$ is nilpotent.
    
    Finally, assume we have a Jordan decomposition $z=z_s+z_n$. By the previous two parts, $f_{x,r}(z_s)$ is semisimple and $f_{x,r}(z_n)$ is nilpotent. To show that this gives a Jordan decomposition of $f_{x,r}(z)$, it suffices to show that $f_{x,r}(z_s)$ and $f_{x,r}(z_n)$ commute. But we have
    \[
    [f_{x,r}(z_s),f_{x,r}(z_n)]=t^rf_{x,r}([z_s,z_n])=0.
    \]
\end{proof}

The map $f_{x,r}$ satisfies a natural compatibility with $q_{x,r}$ which will be useful later.

\begin{lemma}\label{fqcompatibility}
There is a commutative diagram

\begin{center}
\begin{tikzcd}
\mathfrak{k}_{x,r}/\mathfrak{k}_{x,r+} \arrow[r, "q_{x,r}"]\arrow[d, "f_{x,r}"] &C((t))_{=r}\arrow[d]\\
\mathfrak{g}((t))\arrow[r, "q"] & C((t)).
\end{tikzcd}
\end{center}
\end{lemma}

\begin{proof}
By Theorem \ref{depthcenter}, it suffices to show that $q$ sends
\[
\displaystyle\bigoplus_{\langle\alpha,x\rangle+i=r}\mathfrak{h}_{\alpha}^it^i
\]
to $C((t))_{=r}$. We imitate the proof of part 1) of Theorem \ref{depthcenter}. Again, we can reduce to the split case, where we want to show that $q$ sends the subspace
\begin{equation}\label{eq:splitsource}
\displaystyle\bigoplus_{\substack{\alpha\in X^*(T_H)\\ i\in\frac{1}{n}\mathbb{Z} \\ \langle\alpha,y\rangle+i\thinspace=\thinspace r}}\mathfrak{h}_{\alpha}t^i
\end{equation}
to
\[
\displaystyle\bigoplus_{\substack{i\in\mathbb{Z}_+,j\in\frac{1}{n}\mathbb{Z}\\ j\thinspace=\thinspace r\cdot i}}(\mathfrak{h}//H)_it^j.
\]
Choose basises $a_{i,j}$ and $b_{\alpha,j}$ as in the proof of Theorem \ref{depthcenter}. We want to show that the pullbacks of the $a_{i,j}^m$ with $m\neq r\cdot i$ vanish on (\ref{eq:splitsource}). Fix $i,j,$ and $m$, and let the pullback of $a_{i,j}$ along $\mathfrak{h}\rightarrow\mathfrak{h}//H$ equal
\[
\displaystyle\sum_{l}c_{l}b_{\alpha_{l,1},j_{l,1}}b_{\alpha_{l,2},j_{l,2}}\cdots b_{\alpha_{l,i},j_{l,i}}.
\]
As in the proof of Theorem \ref{depthcenter}, every term $b_{\alpha_{l,1},j_{l,1}}^{m_1}b_{\alpha_{l,2},j_{l,2}}^{m_2}\cdots b_{\alpha_{l,i},j_{l,i}}^{m_i}$ appearing in the pullback of $a_{i,j}^m$ satisfies
\[
\displaystyle\sum_{z=1}^i(\langle\alpha_{l,z},y\rangle+r-m_z)=\langle 0,y\rangle+r\cdot i-m \neq 0.
\]
But for such a term to act nontrivially on $(\ref{eq:splitsource})$, we must have $\langle\alpha_{l,z},y\rangle+r-m_z=0$ for every $z$, a contradiction.
\end{proof}

We conclude by noting some quick consequences of the material in this section.

\begin{proposition}\label{finiteness}
The map
\[
(\mathfrak{k}_{x,r}/\mathfrak{k}_{x,r+})//L_x\rightarrow C((t))_{=r}
\]
induced by $q_{x,r}$ is finite.
\end{proposition}

\begin{proof}
We will use the notation from Lemma \ref{qxrdescription} and its preceding discussion. It suffices to show that the composition
\begin{equation}\label{eq:compositionfinite}
(\mathfrak{k}_{x,r}/\mathfrak{k}_{x,r+})//L_x\rightarrow\displaystyle\bigoplus_{i\in\mathbb{Z}_+}C_i((t))^{r\cdot i}\rightarrow\displaystyle\bigoplus_{\substack{i\in\mathbb{Z}_+,j\in\frac{1}{n}\mathbb{Z}\\ j\thinspace=\thinspace r\cdot i}}(\mathfrak{h}//H)_it^j\rightarrow\mathfrak{h}//H
\end{equation}
is finite, as the final two arrows in (\ref{eq:compositionfinite}) are closed embeddings. Lemma \ref{qxrdescription} identifies this composition with the natural map
\[
\mathfrak{h}_r//L_x\rightarrow\mathfrak{h}//H.
\]

Choose an extension of $\mathfrak{c}$ to a Cartan subspace $\mathfrak{c_h}$ of $\mathfrak{h}$. Then the above map fits into the commutative diagram
\begin{center}
\begin{tikzcd}
\mathfrak{c_h} \arrow[r]\arrow[d] & \mathfrak{c}\arrow[d]\\
\mathfrak{h}_r//L_x\arrow[r] & \mathfrak{h}//H.
\end{tikzcd}
\end{center}
The vertical arrows are integral morphisms and the top arrow is a closed embedding. It now follows formally that the bottom arrow is finite, as desired.
\end{proof}
This proposition has an important consequences, which gives another perspective on the depth $r$ locus.
\begin{corollary}\label{unstablevsfiber}
The unstable locus $(\mathfrak{k}_{x,r+}^{\perp}/\mathfrak{k}_{x,r}^{\perp})^{\on{us}}$ can be set-theoretically identified with the fiber $q_{x,r}^{-1}(0).$
\end{corollary}
\begin{proof}
It a result of GIT that the unstable locus coincides with the set-theoretic fiber over $0$ of
\[
\mathfrak{k}_{x,r+}^{\perp}/\mathfrak{k}_{x,r}^{\perp}\rightarrow(\mathfrak{k}_{x,r+}^{\perp}/\mathfrak{k}_{x,r}^{\perp})//L_x.
\]
It thus suffices to show that the set-theoretic fiber over $0$ of
\[
(\mathfrak{k}_{x,r}/\mathfrak{k}_{x,r+})//L_x\rightarrow C((t))_{=r}
\]
consists only of the point $0$.

By Proposition \ref{finiteness}, the set-theoretic fiber must consist of a finite set of points. However, the fiber must also be $\mathbb{G}_m$-invariant, and hence can only contain $0$.
\end{proof}
\subsection{Stratification by twisted Levi subgroups}

We now move to the primary goal of this chapter, which is to relate each semistable $L_x$-orbit in $\mathfrak{k}_{x,r}/\mathfrak{k}_{x,r+}$ to a corresponding orbit for some twisted Levi subgroup. Our starting point is a simple observation. 

\begin{lemma}
The $k((t))$-points of $C$ are in bijection with the conjugacy classes of semisimple elements of $\mathfrak{g}$.
\end{lemma}

\begin{proof}
Recall that we have a tautological map $\mathfrak{g}\rightarrow C$ of $k((t))$-varieties. By definition, this map sends conjugate elements to the same point. It suffices to show that the set of semisimple elements in each fiber forms a single conjugacy class.

Fix a $k((t))$-point $c\in C.$ First we show that there is at least one semisimple element of $\mathfrak{g}$ which maps to $c.$ As Lemma \ref{quasisplitkostant} tells us that there is a section $C\rightarrow\mathfrak{g}$, some $g\in\mathfrak{g}$ maps to $c$. Then if we take a Jordan decomposition of $g$, the semisimple part $g_{\on{ss}}$ also maps to $c$, as desired.

It is well known that two semisimple elements of $\mathfrak{g}\otimes_{k((t))}\overline{k((t))}$ are conjugate if and only if they map to the same point of $C\times_{\Spec k((t))}\Spec\overline{k((t))}.$ It follows that all the semisimple elements of the fiber over $c$ are stably conjugate, i.e., become conjugate over some extension of $k((t))$. Choose a semisimple point $a$ of the fiber over $c$. The set of conjugacy classes in the stable conjugacy class of $a$ can be put in bijection with (see e.g. Section 3 of \cite{Kottwitz}\footnote{Technically, \cite{Kottwitz} studies stable conjugacy on elements of $G$ instead of elements of $\mathfrak{g}.$ However, the proofs also apply to the case of $\mathfrak{g}$ (and in fact simplify, as the centralizer of a semisimple element of $\mathfrak{g}$ is always connected).}) the kernel of the map of Galois cohomology sets \[
H^1(k((t)),Z_G(a))\rightarrow H^1(k((t)),G),
\]
where $Z_G(a)$ is the centralizer of $a$. By Steinberg's theorem \cite{steinberg}, both cohomology sets vanish, so every element stably conjugate to $a$ is in fact conjugate, which gives the desired result.
\end{proof}

In particular, for every $k((t))$-point $c$ of $C$, any two semisimple elements in the fiber over $c$ have conjugate centralizers. As the centralizer of a semisimple element of $\mathfrak{g}\otimes_{k((t))}\overline{k((t))}$ is a Levi subgroup, the centralizer of a semisimple element of $\mathfrak{g}$ is a twisted Levi subgroup. For any nonzero point $p\in C$, let $Z_p\in\TL(G)^{\diamond}$ denote the conjugacy class of the centralizer of a semisimple element mapping to $p$. We will also set $Z_0={\diamond}\in\TL(G)^{\diamond}.$

\begin{lemma}\label{semicontinuity}
Let $G_0$ be an element of $\TL(G)^{\diamond}$, and let $X$ denote the image of $q_{x,r}$, which will be a closed subvariety of $C((t))_{=r}.$ Then the set of points $p\in X$ such that $Z_p\geq G_0$ forms a Zariski-closed subset of $X$. Furthermore, the set of points $p\in X$ where $Z_p=G_0$ is open in this subset.
\end{lemma}

\begin{proof}
The locus $X$ is the image of the finite map of Proposition \ref{finiteness}, so it is closed. Also, the statement is clearly true if $G_0={\diamond}$, so we can assume $G_0$ is in $\TL(G)$ (and thus corresponds to an actual subgroup of $G$.)

The statements of the lemma for $x$ and $g\cdot x$ are equivalent for any $g\in G((t))$, so we can assume that $x\in\mathcal{A}(T)$. Recall that in Proposition \ref{vinbergsummary} and the preceding discussion, we produced a subspace $\mathfrak{c}\subseteq\mathfrak{k}_{x,r}/\mathfrak{k}_{x,r+}$ consisting of semisimple elements and a finite group $W_{\mathfrak{c}}$ acting on $\mathfrak{c}$ such that the map $\mathfrak{c}//W_{\mathfrak{c}}\rightarrow(\mathfrak{k}_{x,r}/\mathfrak{k}_{x,r+})//L_x$ is an isomorphism. Then by Proposition \ref{finiteness}, the composition $\mathfrak{c}\rightarrow\mathfrak{c}//W_{\mathfrak{c}}\rightarrow X$ is finite and surjective.

Lemma \ref{fqcompatibility} tells us that for any element $a$ of $\mathfrak{c}$, $q_{x,r}(a)$, treated as a element of $C((t))$, is equal to $q(f_{x,r}(a))$, so $Z_{q_{x,r}(a)}$ is in the conjugacy class of $Z_{G}(f_{x,r}(a)).$ As the elements of $f_{x,r}(\mathfrak{c})$ commute and are semisimple (by Lemma \ref{semisimpleagreement}), $\mathfrak{g}\otimes_{k((t))}\overline{k((t))}$ decomposes as a sum of $1$-dimensional $f_{x,r}(\mathfrak{c})\otimes_k \overline{k((t))}$-representations $\mathfrak{g}_i$. Each $\mathfrak{g}_i$ corresponds to an element $\beta_i$ in $(f_{x,r}(\mathfrak{c})\otimes_k \overline{k((t))})^*$, and for any $a\in\mathfrak{c}$, $Z_{q_{x,r}(a)}\otimes_{k((t))}\overline{k((t))}$ is the sum of those $\mathfrak{g}_i$ for which $\beta_i$ sends $f_{x,r}(a)$ to $0$.

It follows that the map $a\mapsto Z_{\mathfrak{g}}(f_{x,r}(a))$ is upper semicontinuous, i.e., for any element $g$ of $\mathfrak{g}$, the set of points $a$ of $\mathfrak{c}$ such that $Z_{\mathfrak{g}}(f_{x,r}(a))$ contains $g$ is a Zariski-closed subset of $\mathfrak{c}.$ Because $Z_G(f_{x,r}(a))$ is connected, the function $a\mapsto Z_G(f_{x,r}(a))$ is also upper semicontinuous, and in particular, takes only finitely many values.

Out of the finitely many possible values of $Z_G(f(a))$, let $G'_1,G'_2,\cdots, G'_m$ be the subgroups which contain a conjugate of $G_0$. Then $Z_G(f(a))$ contains a conjugate of $G_0$ if and only if it contains one of the $G'_i.$ It follows that the locus $\mathfrak{c}^{\geq G_0}$ of $a\in\mathfrak{c}$ such that $Z_G(f(a))$ contains a conjugate of $G_0$ is closed. Similarly, the set $\mathfrak{c}^{> G_0}$ of such $a$ such that $Z_G(f(a))$ contains a conjugate of $G_0$ as a proper subgroup is also closed, so the complement of $\mathfrak{c}^{> G_0}$ inside $\mathfrak{c}^{\geq G_0}$ is open inside $\mathfrak{c}^{\geq G_0}$. As the map $\mathfrak{c}\rightarrow X$ is finite, these statements imply the corresponding statements for $X$.
\end{proof}

\begin{remark}
For any fixed $r$, there is some $x$ such that the image of $q_{x,r}$ is the entirety of $C((t))_{=r}.$ This can be seen by choosing $x$ to lie in the building of the centralizer of a "generic" depth $r$ semisimple element. We omit the details, as this is not used in the following.
\end{remark}

Let $(\mathfrak{k}_{x,r}/\mathfrak{k}_{x,r+})^{\geq G_0}$ (resp. $(\mathfrak{k}_{x,r}/\mathfrak{k}_{x,r+})^{=G_0}$) denote the locus of points $p$ in $\mathfrak{k}_{x,r}/\mathfrak{k}_{x,r+}$ where $Z_{q_{x,r}(p)}\geq G_0$ (resp. $Z_{q_{x,r}(p)}=G_0$). By Lemma \ref{semicontinuity}, $(\mathfrak{k}_{x,r}/\mathfrak{k}_{x,r+})^{\geq G_0}$ is a closed sub-locus which contains $(\mathfrak{k}_{x,r}/\mathfrak{k}_{x,r+})^{=G_0}$ as an open subset. We equip them with the induced variety structures. By Corollary \ref{unstablevsfiber}, we have $(\mathfrak{k}_{x,r}/\mathfrak{k}_{x,r+})^{\geq\diamond}=(\mathfrak{k}_{x,r}/\mathfrak{k}_{x,r+})^{=\diamond}=(\mathfrak{k}_{x,r}/\mathfrak{k}_{x,r+})^{\on{us}}.$

We can now turn to the key statement of this paper. Fix $G_0\in\TL(G)$. As we would like to relate the buildings for $G_0$ and $G$, we lift $x$ to a point (which we will also denote by $x$) in the extended building $\widetilde{\mathcal{B}}(G).$ This does not affect the validity of any of the previous results in this section, as the $\mathfrak{k}_{x,r}$ do not change. Now we can (and do) identify $\widetilde{\mathcal{B}}(G'_0)$ with its image in $\widetilde{\mathcal{B}}(G)$ under some embedding, for any conjugate $G'_0$ of $G_0$.

Let $G^x$ be as in Lemma \ref{stabilizerconstruction}, and let $S$ be the set of $G^x$-conjugacy classes of subgroups $G'_0\subseteq G$ which are $G$-conjugate to $G_0$ and satisfy $x\in\widetilde{\mathcal{B}}(G'_0)$. Choose a representative $G_s$ for every $s\in S$, and let $N_s$ be the normalizer of $G_s$ in $G$. There is a natural action of $N_s$ on $\widetilde{\mathcal{B}}(G_s)$, and, as in Lemma \ref{stabilizerconstruction}, we can uniquely construct a reduced subgroup $N_s^x\subseteq N_s$ such that $N_s^x(k((t)))$ is the stabilizer of $x\in\widetilde{\mathcal{B}}(G_s)$ in $N_s(k((t)))$.

We want to relate the $L_x$-orbits in $\mathfrak{k}_{x,r}/\mathfrak{k}_{x,r+}$ to analogous orbits for $G_0$ (or rather, for its conjugates $G_s$). We will use a $G_s$-superscript to denote the $G_s$-counterpart of each of our structures for $G$. So for instance, we have a Moy-Prasad subspace $\mathfrak{k}_{x,r}^{G_s}\subseteq\mathfrak{g}_s((t))$, a locus $(\mathfrak{k}_{x,r}^{G_s}/\mathfrak{k}_{x,r+}^{G_s})^{=G_1}$ for any $G_1\in\TL(G_s)^{\diamond},$ a map $q_{x,r}^{G_s}:\mathfrak{k}_{x,r}^{G_s}/\mathfrak{k}_{x,r+}^{G_s}\rightarrow C((t))_{=r},$ etc.

For every $s\in S$, there is a closed embedding
\[
i_s:\mathfrak{k}_{x,r}^{G_s}/\mathfrak{k}_{x,r+}^{G_s}\rightarrow\mathfrak{k}_{x,r}/\mathfrak{k}_{x,r+}.
\]
We wish to understand how this map interacts with the stratification defined earlier in this section.

\begin{lemma}\label{levistabilizercompatibility}
Let $p_s$ be a nonzero $k$-point of $\mathfrak{k}_{x,r}^{G_s}/\mathfrak{k}_{x,r+}^{G_s}.$ Then $Z_{q_{x,r}(i_s(p_s))}$ contains a conjugate of $Z^{G_s}_{q_{x,r}^{G_s}(p_s)}.$
\end{lemma}

\begin{proof}
Consider the natural map $C^{G_s}\cong\mathfrak{g}_s//G_s\rightarrow\mathfrak{g}//G\cong C.$

We claim that the induced map $C^{G_s}((t))\rightarrow C((t))$ maps $C^{G_s}((t))_{\geq r}$ into $C((t))_{\geq r}$ and that the composition
\[
C^{G_s}((t))_{\geq r}\rightarrow C((t))_{\geq r}\rightarrow C((t))_{\geq r}/C((t))_{>r}\cong C((t))_{=r}
\]
factors through $C^{G_s}((t))_{\geq r}/C^{G_s}((t))_{>r}\cong C^{G_s}((t))_{=r}.$

As the bigrading with components $C_i((t))^j$ is naturally compatible with extensions of $k((t))$, it suffices to treat the case where $G_s$ and $G$ are split, and $G_s$ is a Levi subgroup inside $G$. In this case, there exists an identification of the inclusion $G_s\subseteq G$ with $L\times_{\Spec k}\Spec k((t))\subseteq H\times_{\Spec k}\Spec k((t))$ for some Levi subgroup $L$ of $H$. The map $C^{G_s}((t))\rightarrow C((t))$ can then be identified with $(\mathfrak{l}//L)((t))\rightarrow(\mathfrak{h}//H)((t)),$ and the desired statement follows from a short computation.

Now, using Theorem \ref{depthcenter}, we have a cube
\begin{center}
\begin{tikzcd}[row sep={40,between origins}, column sep={40,between origins}]
      & \mathfrak{k}_{x,r}^{G_s} \ar{rr}\ar{dd}\ar{dl} & & C^{G_s}((t))_{\geq r} \ar{dd}\ar{dl} \\
    \mathfrak{k}_{x,r}^{G_s}/\mathfrak{k}_{x,r+}^{G_s} \ar[crossing over]{rr} \ar{dd} & & C^{G_s}((t))_{=r} \\
      & \mathfrak{k}_{x,r}  \ar{rr} \ar{dl} & &  C((t))_{\geq r} \ar{dl} \\
    \mathfrak{k}_{x,r}/\mathfrak{k}_{x,r+} \ar{rr} && C((t))_{=r}. \ar[from=uu,crossing over]
\end{tikzcd}
\end{center}
We know the commutativity of every face of this cube, except for the front face. As $\mathfrak{k}_{x,r}^{G_s}\rightarrow\mathfrak{k}_{x,r}^{G_s}/\mathfrak{k}_{x,r+}^{G_s}$ is an epimorphism, it follows that the front face is commutative, and so the map $C^{G_s}((t))_{=r}\rightarrow C((t))_{=r}$ sends $q_{x,r}^{G_s}(p_s)$ to $q_{x,r}(i_s(p_s))$. Therefore, for any semisimple element $g_s\in\mathfrak{g}_s$ whose image in $C^{G_s}$ is equal to $q_{x,r}^{G_s}(p_s)$, the map $\mathfrak{g}_s\rightarrow\mathfrak{g}\rightarrow C$ sends $g_s$ to $q_{x,r}(i_s(p_s))$. It follows that $Z_{q_{x,r}(i_s(p_s))}$ is conjugate to $Z_G(g_s)$ while $Z^{G_s}_{q_{x,r}^{G_s}(p_s)}$ is conjugate to $Z_{G_s}(g_s)$, which implies the statement of the lemma.
\end{proof}

In particular, $i_s$ must map $(\mathfrak{k}_{x,r}^{G_s}/\mathfrak{k}_{x,r+}^{G_s})^{=G_s}$ into $(\mathfrak{k}_{x,r}/\mathfrak{k}_{x,r+})^{\geq G_0}.$ We would like to isolate the sublocus of $(\mathfrak{k}_{x,r}^{G_s}/\mathfrak{k}_{x,r+}^{G_s})^{=G_s}$ which maps into $(\mathfrak{k}_{x,r}/\mathfrak{k}_{x,r+})^{= G_0}.$ Let $Z_s$ be the center of $G_s$. Then the decomposition $\mathfrak{g}_s\cong[\mathfrak{g}_s,\mathfrak{g}_s]\oplus\mathfrak{z}_s$ induces a decomposition
\[
\mathfrak{k}_{x,r}^{G_s}/\mathfrak{k}_{x,r+}^{G_s}\cong\mathfrak{k}_{x,r}^{[G_s,G_s]}/\mathfrak{k}_{x,r+}^{[G_s,G_s]}\oplus\mathfrak{k}_{r}^{Z_s}/\mathfrak{k}_{r+}^{Z_s},
\]
where $\mathfrak{k}_{x,r}^{[G_s,G_s]}$ (resp. $\mathfrak{k}_{r}^{Z_s}$) are the Moy-Prasad subgroups associated to $[G_s,G_s]$ (resp. $Z_s$). (We note that the building of $Z_s$ is just a single point, hence the lack of dependence of $\mathfrak{k}_{r}^{Z_s}$ on $x$, while the building of $[G_s,G_s]$ coincides with that of $G$.)

The map $q_{x,r}^{G_s}$ also decomposes as a sum of the maps
\[
q_{x,r}^{[G_s,G_s]}:\mathfrak{k}_{x,r}^{[G_s,G_s]}/\mathfrak{k}_{x,r+}^{[G_s,G_s]}\rightarrow C^{[G_s,G_s]}((t))_{=r}
\]
and
\[
q_{x,r}^{Z_s}:\mathfrak{k}_r^{Z_s}/\mathfrak{k}_{r+}^{Z_s}\rightarrow C^{Z_s}((t))_{=r}.
\]
Furthermore, unwinding the definitions, we see that $C^{Z_s}((t))_{\leq r}$ can be identified with $\mathfrak{k}_r^{Z_s}$, and so $q_{x,r}^{Z_s}$ is an isomorphism.

A semisimple element of $\mathfrak{g}_s$ has centralizer $G_s$ if and only if it lies in $\mathfrak{z}_s$. It follows that $(\mathfrak{k}_{x,r}^{G_s}/\mathfrak{k}_{x,r+}^{G_s})^{=G_s}$ can be identified with 
\[
(q_{x,r}^{G_s})^{-1}(C^{Z_s}((t))_{=r})\cong (q_{x,r}^{[G_s,G_s]})^{-1}(0)\oplus(C^{Z_s}((t))_{=r}\backslash 0).
\]
Let $C^{Z_s}((t))_{=r,\on{gen}}$ be the locus of points $p$ in $C^{Z_s}((t))_{=r}$ whose image along 
\[
C^{Z_s}((t))_{=r}\rightarrow C^{Z_s}((t))\cong\mathfrak{z}_s((t))\rightarrow\mathfrak{g}((t))
\]
has stabilizer $G_s$ with respect to the adjoint $G$-action. (As $p$ comes from $Z_s$, the stabilizer must contain $G_s$.) We claim that $C^{Z_s}((t))_{=r,\on{gen}}$ is the complement of a finite number of hyperplanes in $C^{Z_s}((t))_{=r}$ and hence is naturally an open $k$-subvariety.

To see this, note that for any positive integer $i$, $C^{Z_s}((t))_{=r,\on{gen}}$ is the intersection of $C^{Z_s}((t))_{=r}$ with $C^{Z_s}((t^{\frac{1}{i}}))_{=r,\on{gen}}$ inside $C^{Z_s}((t^{\frac{1}{i}}))_{=r}$. It thus suffices to prove the desired statement after an extension of $k((t))$. We can then assume that $Z_s$ is a split torus. 

In this case, the adjoint action of $Z_s$ on $\mathfrak{g}$ gives a decomposition
\[
\mathfrak{g}=\displaystyle\bigoplus_{\gamma\in X^*(Z_s)}\mathfrak{g}_{\gamma}.
\]
It is easily checked that $p\in C^{Z_s}((t))_{=r}$ is in $C^{Z_s}((t))_{=r,\on{gen}}$ iff $p$ does not pair to zero with any nonzero $\gamma\in X^*(Z_s)$ such that $\mathfrak{g}_{\gamma}$ is nontrivial. For any such $\gamma$, the set of $p\in C^{Z_s}((t))_{=r}$ which pairs to zero with $\gamma$ is a hyperplane, as desired.

Now denote $(q_{x,r}^{[G_s,G_s]})^{-1}(0)\oplus C^{Z_s}((t))_{=r,\on{gen}}$ by $(\mathfrak{k}_{x,r}^{G_s}/\mathfrak{k}_{x,r+}^{G_s})^{=G_s,\on{gen}}.$ The proof of Lemma \ref{levistabilizercompatibility} shows that $i_s$ maps $(\mathfrak{k}_{x,r}^{G_s}/\mathfrak{k}_{x,r+}^{G_s})^{=G_s,\on{gen}}$ into $(\mathfrak{k}_{x,r}/\mathfrak{k}_{x,r+})^{=G_0}$. This allows us to finally state the main theorem of this paper.

\begin{theorem}\label{basecase}
The map
\[
\displaystyle\bigcup_{s\in S}(\mathfrak{k}_{x,r}^{G_s}/\mathfrak{k}_{x,r+}^{G_s})^{=G_s,\on{gen}}/(N^x_s/K_{x,0+}^{G_s})\rightarrow(\mathfrak{k}_{x,r}/\mathfrak{k}_{x,r+})^{=G_0}/(G^x/K_{x,0+})
\]
is an isomorphism of Artin stacks.
\end{theorem}

We will give the proof as a separate section.

\section{Proof of Theorem \ref{basecase}}\label{s:proof}

Note that the theorem can be rephrased as saying that the map of varieties
\[
\Omega:\displaystyle\bigcup_{s\in S}(G^x/K_{x,0+})\times^{(N^x_s/K_{x,0+}^{G_s})}(\mathfrak{k}_{x,r}^{G_s}/\mathfrak{k}_{x,r+}^{G_s})^{=G_s,\on{gen}}\rightarrow(\mathfrak{k}_{x,r}/\mathfrak{k}_{x,r+})^{=G_0}
\]
is an isomorphism. This is the statement we will now prove.

\subsection{Bijectivity on semisimple elements}

We start with the following lemma, which will let us define a pointwise inverse to $\Omega$ on semisimple elements.

\begin{lemma}\label{inverseconstruction}
Let $\overline{g}$ be any nonzero semisimple element of $\mathfrak{k}_{x,r}/\mathfrak{k}_{x,r+}$. Then:
\begin{enumerate}
\item There exists a semisimple element $g\in\mathfrak{k}_{x,r}$ lifting $\overline{g}$ such that $q(g)=q_{x,r}(\overline{g})$.
\item All choices of $g$ as in part (1) are $K_{x,0+}$-conjugate.
\item For any such $g$, we have $x\in\widetilde{B}(Z_G(g))$, and the equality
\begin{equation}\label{eq:centralizerdecomp}
\mathfrak{k}_{x,s}/\mathfrak{k}_{x,s+}\cong[\overline{g},\mathfrak{k}_{x,s-r}/\mathfrak{k}_{x,(s-r)+}]\oplus\mathfrak{k}^{Z_G(g)}_{x,s}/\mathfrak{k}^{Z_G(g)}_{x,s+}
\end{equation}
holds for any rational number $s$. Furthermore, $\mathfrak{k}^{Z_G(g)}_{x,s}/\mathfrak{k}^{Z_G(g)}_{x,s+}$ coincides with the kernel of $\on{ad}\overline{g}:\mathfrak{k}_{x,s}/\mathfrak{k}_{x,s+}\rightarrow\mathfrak{k}_{x,r+s}/\mathfrak{k}_{x,(r+s)+}$.
\end{enumerate}
\end{lemma}

\begin{proof}
First we show part (1). We can assume without loss of generality that $x\in\widetilde{A}(T).$ Then by Lemma \ref{fqcompatibility}, $g=f_{x,r}(\overline{g})$ satisfies the desired properties.

Next we show part (3), assuming part (2). Again assume that $x\in\widetilde{A}(T).$ By part (2), it suffices to show that the conditions of part (3) are satisfied for the specific choice $g=f_{x,r}(\overline{g})$. We note that it suffices to prove them after extending of $k((t))$, as the original statements can be recovered by taking the fixed points under the Galois action.

Because $(\mathfrak{k}_{x,r}/\mathfrak{k}_{x,r+})$ has a nonzero semisimple element, $r$ must be a rational number, and after potentially replacing $k((t))$ by an extension we can assume that $r$ is an integer and $G$ is split (so we can identify it with $H\times_{\Spec k}\Spec k((t))$). In this case, $\mathfrak{k}_{x,r}/\mathfrak{k}_{x,r+}$ contains $t^r\cdot\mathfrak{t}[[t]]/t^{r+1}\cdot\mathfrak{t}[[t]]$ as a Cartan subspace. Acting on $\overline{g}$ by an element of $L_x\subseteq G^x/K_{x,0+}$, we can assume that $\overline{g}\in t^r\cdot\mathfrak{t}[[t]]/t^{r+1}\cdot\mathfrak{t}[[t]].$ It then follows that $Z_G(f_{x,r}(\overline{g}))$ contains $T\times_{\Spec k}\Spec k((t))$, and hence, $\widetilde{\mathcal{B}}(Z_G(f_{x,r}(\overline{g})))$ contains $\widetilde{\mathcal{A}}(T)$, which contains $x$, as desired.

Let us show that (\ref{eq:centralizerdecomp}) holds in this case. Each term is a sum of affine root spaces. Let $\mathfrak{g}_{\alpha}t^m$ be an affine root space appearing in the decomposition of $\mathfrak{k}_{x,s}/\mathfrak{k}_{x,s+}$. Then it is easy to check that $\mathfrak{g}_{\alpha}t^m$ appears in $[\overline{g},\mathfrak{k}_{x,s-r}/\mathfrak{k}_{x,(s-r)+}]$ iff the element in $\mathfrak{t}$ corresponding to $\overline{g}$ does not pair with $\alpha$ to zero. On the other hand $\mathfrak{g}_{\alpha}t^m$ appears in $\mathfrak{k}^{Z_G(g)}_{x,s}/\mathfrak{k}^{Z_G(g)}_{x,s+}$ iff the value of this pairing is zero, so we get the desired direct sum decomposition. The same computation identifies $\mathfrak{k}^{Z_G(g)}_{x,s}/\mathfrak{k}^{Z_G(g)}_{x,s+}$ with the kernel of $\on{ad}\overline{g}.$

Finally, we treat part (2). Let $g$ be a lift of $\overline{g}$ satisfying the conditions in both part (1) and (3) (such an element was proven to exist above). Let $g_1$ be another lift of $\overline{g}$ satisfying the conditions of part (1). We claim that there is a $K_{x,0+}$-conjugate of $g_1$ which commutes with $g$. To see this, we construct an element $g(s)\in K_{x,0+}/K_{x,(s-r)+}$ for every real number $s\geq r$ with $\mathfrak{k}_{x,s}\neq\mathfrak{k}_{x,s+}$. This series of elements will satisfy two properties:
\begin{enumerate}
\item The $g(s)$-conjugate of the image of $g_1$ in $\mathfrak{k}_{x,r}/\mathfrak{k}_{x,s+}$ lies in $\mathfrak{k}^{Z_G(g)}_{x,r}/\mathfrak{k}^{Z_G(g)}_{x,s+}$.
\item For $s_2>s_1$, $g(s_1)$ is the image of $g(s_2)$ in $K_{x,0+}/K_{x,(s_1-r)+}$.
 \end{enumerate}
 Then when we conjugate $g_1$ by the inverse limit of the $g(s)$, we will get an element of $Z_{\mathfrak{g}}(g),$ as desired.

We can set $g(r)=1$. To construct the rest of the $g(s)$, we use induction. Let $s>r$ be a real number such that $\mathfrak{k}_{x,s}\neq\mathfrak{k}_{x,s+}$, and let $s_1$ be the largest number less than $s$ with $\mathfrak{k}_{x,s_1}\neq\mathfrak{k}_{x,s_1+}$. Assume that a value of $g(s_1)$ satisfying the properties has already been constructed. Then choose a lift $g(s)_0\in K_{x,0+}/K_{x,(s-r)+}$ lifting $g(s_1)$. By the inductive hypothesis, we will have
\[
\on{ad} g(s)_0(\overline{g}_1)\in\mathfrak{k}^{Z_G(g)}_{x,r}/\mathfrak{k}^{Z_G(g)}_{x,s+}+\mathfrak{k}_{x,s}/\mathfrak{k}_{x,s+},
\]
where $\overline{g}_1$ is the image of $g_1$ in $\mathfrak{k}_{x,r}/\mathfrak{k}_{x,s+}$. By part (3), the right hand side can be rewritten as $\mathfrak{k}^{Z_G(g)}_{x,r}/\mathfrak{k}^{Z_G(g)}_{x,s+}\oplus[\overline{g},\mathfrak{k}_{x,s-r}/\mathfrak{k}_{x,(s-r)+}]$. Let $g(s)_1$ be an element of $K_{x,s-r}/K_{x,(s-r)+}\cong\mathfrak{k}_{x,s-r}/\mathfrak{k}_{x,(s-r)+}$ such that $[\overline{g},g(s)_1]$ coincides with the image of $\on{ad} g(s)_0(\overline{g}_1)$ in $[\overline{g},\mathfrak{k}_{x,s-r}/\mathfrak{k}_{x,(s-r)+}]$. Then a short computation shows that $g(s)=g(s)_1g(s)_0$ satisfies the desired properties.

Let us replace $g_1$ with its $K_{x,0+}$-conjugate which commutes with $g$. We wish to show that $g_1=g.$ First, we claim that the images $c_1$ and $c$ of $g_1$ and $g,$ respectively, in $C^{Z_G(g)}\cong Z_{\mathfrak{g}}(g)//Z_G(g)$ agree. By Theorem \ref{depthcenter}, $c_1$ and $c$ are elements of $C^{Z_G(g)}_{\geq r}$ which are equal modulo $C^{Z_G(g)}_{>r}.$ Furthermore, by assumption, the images of $c_1$ and $c$ in $C$ are equal to $q_{x,r}(\overline{g})$.

It suffices to show that $c_1=c$ after an extension of $k((t))$, so we can assume that $r$ is an integer, $G$ and $Z_G(g)$ are split, and $Z_G(g)$ is a Levi subgroup of $G$. Then the bigradings on $C((t))$ and $C^{Z_G(g)}((t))$ with components $C_i((t))^j$ and $C^{Z_G(g)}_i((t))^j$, respectively, come from $\mathbb{G}_m\times\mathbb{G}_m$-actions, and the map $C^{Z_G(g)}((t))\rightarrow C((t))$ is $\mathbb{G}_m\times\mathbb{G}_m$-equivariant. Consider the inclusion $a:\mathbb{G}_m\rightarrow\mathbb{G}_m\times\mathbb{G}_m$ corresponding to the cocharacter $(r,-1)\in\mathbb{Z}\oplus\mathbb{Z}\cong X_*(\mathbb{G}_m\times\mathbb{G}_m).$ The fixed points of this action on $C((t))$ (resp. $C^{Z_G(g)}((t))$) are $C((t))_{=r}$ (resp. $C^{Z_G(g)}((t))_{=r}$.)

In particular, for any $z\in\mathbb{G}_m$, the images of $a(z)\cdot c_1$ and $c$ in $C((t))$ agree. However, if $c_1\neq c$, then $c_1$ cannot lie in $C^{Z_G(g)}((t))_{=r}$, and so this gives us infinitely many points of $C^{Z_G(g)}((t))$ which map to the same point of $C((t))$. But $C^{Z_G(g)}\rightarrow C$ is finite, so we get a contradiction, and $c_1$ must equal $c$.

Therefore, $g_1$ and $g$ must become $Z_G(g)$-conjugate after some extension of $k((t))$. But $g$ is central in $Z_G(g)$, so $g_1=g$, as desired.

\end{proof}

Lemma \ref{inverseconstruction} allows us to associate to any semisimple $\overline{g}\in(\mathfrak{k}_{x,r}/\mathfrak{k}_{x,r+})^{=G_0}$ an element $s$ of $S$, by taking the centralizer of a lift $g$ satisfying the conditions of part (1). Assume that $Z_G(g)=uG_su^{-1},$ for $u\in G^x,$ and let $\overline{u}$ be the image of $u$ in $G^x/K_{x,0+}.$ Then $Z_G(u^{-1}g)=G_s$, so $u^{-1}g$ lies in $\mathfrak{k}^{G_s}_{x,r}$, and its image in $\mathfrak{k}^{G_s}_{x,r}/\mathfrak{k}^{G_s}_{x,r+}$ is equal to $\overline{u}^{-1}\overline{g}$. Then our pointwise inverse to $\Omega$ sends $\overline{g}$ to
\[
(\overline{u}, \overline{u}^{-1}\overline{g})\in(G^x/K_{x,0+})\times^{(N^x_s/K_{x,0+}^{G_s})}(\mathfrak{k}_{x,r}^{G_s}/\mathfrak{k}_{x,r+}^{G_s})^{=G_s,\on{gen}}.
\]
This is well defined because $u$ is well-defined up to right multiplication by an element of $N_s^x$, and it is easy to check that this is indeed a pointwise inverse on semisimple elements.

We conclude with a slight generalization of Lemma \ref{inverseconstruction} which will not be needed in the rest of this paper (but which is useful for some applications.)

\begin{lemma}\label{inverseconstruction2}
Let $\overline{v}_i\in\mathfrak{k}_{x,r_i}/\mathfrak{k}_{x,r_i+}, i=1,2,\cdots, m$ be nonzero semisimple elements of which pairwise commute. Then:
\begin{enumerate}
\item There exist semi-simple lifts $v_i\in\mathfrak{k}_{x,r_i}$ of $\overline{v}_i, i=1,2,\cdots, m$ which pairwise commute and satisfy $q(v_i)=q_{x,r_i}(\overline{v}_i)$ for all $i$.
\item All choices of $(v_1,v_2,\cdots, v_m)$ as in part (1) are $K_{x,0+}$-conjugate.
\item For any choice of the $v_i$, we have $x\in\widetilde{B}(Z_G(v_1,v_2,\cdots, v_m)).$
 \end{enumerate}
\end{lemma}

\begin{proof}
Induct on $m$. The statement is tautological for $m=0$. For $m>0$, Lemma \ref{inverseconstruction} shows that there exists a lift $v_1$ of $\overline{v}_1$ such that $q(v_1)=q_{x,r_1}(\overline{v}_1)$ and that furthermore, all such lifts are $K_{x,0+}$-conjugate. Part (3) of Lemma \ref{inverseconstruction} shows that $x\in\widetilde{B}(Z_G(v_1))$ and that the $\overline{v}_i$ can be identified with elements of $\mathfrak{k}_{x,r_i}^{Z_G(v_1)}/\mathfrak{k}_{x,r_i+}^{Z_G(v_1)}.$ The desired statements now follow from applying the inductive hypothesis with the group $G$ replaced by $Z_G(v_1)$, as long as we know that for a lift $v_i\in Z_G(v_1)$ of $\overline{v}_i$, $q(v_i)=q_{x,r_i}(\overline{v}_i)$ is equivalent to $q^{Z_G(v_1)}(v_i)=q_{x,r_i}^{Z_G(v_1)}(\overline{v}_i).$

As in the proof of Lemma \ref{levistabilizercompatibility}, we have a commutative cube
\begin{center}
\begin{tikzcd}[row sep={40,between origins}, column sep={50,between origins}]
      & \mathfrak{k}_{x,r_i}^{Z_G(v_1)} \ar{rr}\ar{dd}\ar{dl} & & C^{Z_G(v_1)}((t))_{\geq r_i} \ar{dd}\ar{dl} \\
    \mathfrak{k}_{x,r_i}^{Z_G(v_1)}/\mathfrak{k}_{x,r_i+}^{Z_G(v_1)} \ar[crossing over]{rr} \ar{dd} & & C^{Z_G(v_1)}((t))_{=r_i} \\
      & \mathfrak{k}_{x,r_i}  \ar{rr} \ar{dl} & &  C((t))_{\geq r_i} \ar{dl} \\
    \mathfrak{k}_{x,r_i}/\mathfrak{k}_{x,r_i+} \ar{rr} && C((t))_{=r_i}. \ar[from=uu,crossing over]
\end{tikzcd}
\end{center}
This cube immediately implies that if $q^{Z_G(v_1)}(v_i)=q_{x,r_i}^{Z_G(v_1)}(\overline{v}_i),$ then $q(v_i)=q_{x,r_i}(\overline{v}_i).$ Let us show the converse.

It suffices to show that $q^{Z_G(v_1)}(v_i)$ lies in $C^{Z_G(v_1)}((t))_{=r_i}.$ Since $q(v_i)=q_{x,r_i}(\overline{v}_i),$ we can deduce that the image of $q^{Z_G(v_1)}(v_i)$ in $C((t))_{\geq r_i}$ lies in $C((t))_{=r_i}.$ Thus, we can reduce to showing that the (set-theoretic) preimage of $C((t))_{=r_i}$ along $C^{Z_G(v_1)}((t))_{\geq r_i}\rightarrow C((t))_{\geq r_i}$ is equal to $C^{Z_G(v_1)}((t))_{=r_i}$.

Again as in the proof of Lemma \ref{levistabilizercompatibility}, we can pass to an extension of $k((t))$ and reduce to the case where the inclusion $Z_G(v_1)\subseteq G$ can be identified with an inclusion $L\times_{\Spec k}\Spec k((t))\subseteq H\times_{\Spec k}\Spec k((t))$ for some Levi subgroup $L$ of $H$. The map $C^{Z_G(v_1)}((t))\rightarrow C((t))$ is then identified with $(\mathfrak{l}//L)((t))\rightarrow (\mathfrak{h}//H)((t))$. The natural bigradings on $(\mathfrak{l}//L)((t))$ and $(\mathfrak{h}//H)((t))$ come from $\mathbb{G}_m\times\mathbb{G}_m$-actions, and the map $(\mathfrak{l}//L)((t))\rightarrow (\mathfrak{h}//H)((t))$ is $\mathbb{G}_m\times\mathbb{G}_m$-equivariant.

Since $\overline{v}_i$ is nonzero, $r$ must be rational by Lemma \ref{rationality}. Write $r=\frac{a}{b}$, where $a$ and $b$ are integers. Then $(\mathfrak{h}//H)((t))_{=r}$ (resp. $(\mathfrak{l}//L)((t))_{=r}$) is the fixed locus of the linear $\mathbb{G}_m$ action on $(\mathfrak{h}//H)((t))$ (resp. $(\mathfrak{l}//L)((t))$) where $z\in\mathbb{G}_m$ acts on $(\mathfrak{h}//H)_i((t))^j$ (resp. $(\mathfrak{l}//L)_i((t))^j$) by multiplication by $z^{b\cdot j-a\cdot i}$. Furthermore, as $\mathfrak{l}//L\rightarrow\mathfrak{h}//H$ is finite, the preimage in $(\mathfrak{l}//L)((t))$ of any $k$-point $p$ of $(\mathfrak{h}//H)((t))$ must be finite. If $p$ lies in $(\mathfrak{h}//H)((t))_{=r}$, then its preimage must also be $\mathbb{G}_m$-invariant and thus consist only of points in the $\mathbb{G}_m$-fixed locus $(\mathfrak{l}//L)((t))_{=r},$ as desired.

\end{proof}

\subsection{Bijectivity on all elements}
To extend the results of the previous subsection, we will need to understand the centralizers in $\mathfrak{k}_{x,r}/\mathfrak{k}_{x,r+}$ of semisimple elements of $\mathfrak{k}_{x,r}/\mathfrak{k}_{x,r+})^{=G_0}$. We start with the case where $x\in\widetilde{\mathcal{A}}(T)$.

\begin{lemma}\label{finiteaffinecomparison}
Assume that $x\in\widetilde{\mathcal{A}}(T)$, and equip $\mathfrak{h}$ with the grading constructed from $on \ref{ss:grading}.x$ in Subsecti Let $\overline{g}$ be a semisimple element of $(\mathfrak{k}_{x,r}/\mathfrak{k}_{x,r+})^{=G_0},$ and let $h\in\mathfrak{h}_r$ be the element corresponding to $\overline{g}$ under the isomorphism
\[
\mathfrak{h}_r\cong \mathfrak{k}_{x,r}/\mathfrak{k}_{x,r+}.
\]
Then $\on{dim}_k Z_H(h)=\on{dim}_{k((t))}G_0.$
\end{lemma}

\begin{proof}
Let $\mathfrak{c}$ be a Cartan subspace of $\mathfrak{h}_r$ containing $h$. Then, as in the proof of Lemma \ref{semicontinuity}, write $\mathfrak{g}\otimes_{k((t))}\overline{k((t))}$ as a sum of $1$-dimensional $f_{x,r}(\mathfrak{c})\otimes_k\overline{k((t))}$-representations $\mathfrak{g}_i$, and assume that $\mathfrak{g}_i$ corresponds to the character $\beta_i\in(f_{x,r}(\mathfrak{c})\otimes_k\overline{k((t))})^*.$ Again, $Z_{\overline{g}}\otimes_{k((t))}\overline{k((t))}$ will be conjugate to the sum of the $\mathfrak{g}_i$ with $\beta_i(f_{x,r}(h))=0.$ Similarly, $\mathfrak{h}$ decomposes as a sum of $1$-dimensional $\mathfrak{c}$-representations, each of which corresponds to a character $\gamma_i\in\mathfrak{c}^*.$

As $\overline{g}$ is non-zero and semisimple, Proposition \ref{rationality} lets us assume that $r$ is rational. We may also assume that $x$ has rational coordinates, i.e., $x\in X_*(T)\otimes\mathbb{Q}$, which will assure that $\mathfrak{h}_s$ can be nonzero only for $s\in\mathbb{Q}$. We can then define an isomorphism of Lie algebras
\[
f:\mathfrak{h}\otimes_k\overline{k((t))}\rightarrow\mathfrak{g}\otimes_{k((t))}\overline{k((t))}
\]
which restricts to $t^{-r}\cdot f_{x,r}$ on $\mathfrak{h}_r$.

The map $f$ induces an identification between $\mathfrak{c}\otimes_k\overline{k((t))}$ and $f_{x,r}(\mathfrak{c})\otimes_k\overline{k((t))},$ under which each $\gamma_i$ must correspond to some $\beta_j$. Fix our indexing so that $\gamma_i$ corresponds to $\beta_i$. Then we see that $\gamma_i(h)=0$ iff $\beta_i(f(h))=\beta_i(t^{-r}f_{x,r}(\overline{g}))=t^{-r}\beta_i(f_{x,r}(\overline{g}))=0.$ It follows that $f(Z_{\mathfrak{h}}(h)\otimes_k\overline{k((t))})=Z_{\mathfrak{g}}(f_{x,r}(\overline{g}))\otimes_{k((t))}\overline{k((t))}.$ The desired equality of dimensions follows.
\end{proof}

\begin{corollary}\label{finiteaffinecomparison2}
Let $\overline{g}_s$ be a semisimple element of $(\mathfrak{k}_{x,r}^{G_s}/\mathfrak{k}_{x,r+}^{G_s})^{=G_s,\on{gen}}\subseteq\mathfrak{k}_{x,r}/\mathfrak{k}_{x,r+}$. Then the centralizer of $\overline{g}_s$ in $\mathfrak{k}_{x,r}/\mathfrak{k}_{x,r+}$ is equal to $\mathfrak{k}_{x,r}^{G_s}/\mathfrak{k}_{x,r+}^{G_s}$.
\end{corollary}

\begin{proof}
Translating by an element of $G(k((t))$, we can assume that $x\in\widetilde{A}(T).$ For any $r'\in\mathbb{R}$, the kernel of
\begin{equation}\label{eq:adgrading}
\on{ad}\overline{g}_s:\mathfrak{k}_{x,r'}/\mathfrak{k}_{x,r'+}\rightarrow\mathfrak{k}_{x,r+r'}/\mathfrak{k}_{x,(r+r')+}
\end{equation}
contains $\mathfrak{k}_{x,r'}^{G_s}/\mathfrak{k}_{x,r'+}^{G_s}$. Under the isomorphism
\[
\mathfrak{h}\cong\displaystyle\bigoplus_{r'\in\mathbb{R}/\mathbb{Z}}\mathfrak{k}_{x,r'}/\mathfrak{k}_{x,r'+}
\]
defined by Subsection \ref{ss:grading}, the centralizer of $\overline{g}_s$ contains
\[
\displaystyle\bigoplus_{r'\in\mathbb{R}/\mathbb{Z}}\mathfrak{k}_{x,r'}^{G_s}/\mathfrak{k}_{x,r'+}^{G_s},
\]
which is of dimension $\dim G_s=\dim G_0.$ Lemma \ref{finiteaffinecomparison} tells us that this containment is an equality, so the kernel of (\ref{eq:adgrading}) must equal $\mathfrak{k}_{x,r'}^{G_s}/\mathfrak{k}_{x,r'+}^{G_s}$. Setting $r'=r$, we get the desired statement.
\end{proof}

Let $(\mathfrak{k}_{x,r}/\mathfrak{k}_{x,r+})^{=G_0}(k)^{\on{ss}}$ (resp. $(\mathfrak{k}_{x,r}^{G_s}/\mathfrak{k}_{x,r+}^{G_s})^{=G_s,\on{gen}}(k)^{\on{ss}}$) denote the set of semisimple $k$-points of $(\mathfrak{k}_{x,r}/\mathfrak{k}_{x,r+})^{=G_0}$ (resp. ($\mathfrak{k}_{x,r}^{G_s}/\mathfrak{k}_{x,r+}^{G_s})^{=G_s,\on{gen}}$). We have a commutative square of sets of $k$-points
\begin{center}
\begin{tikzcd}
\displaystyle\bigcup_{s\in S}(G^x/K_{x,0+})(k)\times^{(N^x_s/K_{x,0+}^{G_s})(k)}(\mathfrak{k}_{x,r}^{G_s}/\mathfrak{k}_{x,r+}^{G_s})^{=G_s,\on{gen}}(k)\arrow[r, "\Omega(k)"] \arrow{d} & (\mathfrak{k}_{x,r}/\mathfrak{k}_{x,r+})^{=G_0}(k) \arrow{d}\\
\displaystyle\bigcup_{s\in S}(G^x/K_{x,0+})(k)\times^{(N^x_s/K_{x,0+}^{G_s})(k)}(\mathfrak{k}_{x,r}^{G_s}/\mathfrak{k}_{x,r+}^{G_s})^{=G_s,\on{gen}}(k)^{\on{ss}}\arrow[r] & (\mathfrak{k}_{x,r}/\mathfrak{k}_{x,r+})^{=G_0}(k)^{\on{ss}},
\end{tikzcd}
\end{center}
where the vertical arrows consist of taking semisimple parts\footnote{A priori, we do not know that this operation defines a map of varieties, hence why we currently work with the above square only at the level of $k$-points.}. In the previous subsection, we showed that the bottom arrow is an isomorphism. It will not be hard to upgrade this to all of $\Omega(k)$.

\begin{proposition}\label{bijectivity}
The map $\Omega$ is bijective on $k$-points.
\end{proposition}

\begin{proof}
Because we know that the bottom arrow is an isomorphism, it suffices to show that $\Omega(k)$ induces bijections on fibers of the vertical maps. As all maps in the above square are $(G^x/K_{x,0+})(k)$-equivariant, it suffices to show that for any $s\in S$ and semisimple element $\overline{g}_s\subseteq(\mathfrak{k}_{x,r}^{G_s}/\mathfrak{k}_{x,r+}^{G_s})^{=G_s,\on{gen}},$ $\Omega$ induces a bijection on the fibers above $1\times\overline{g}_s.$ 

The fiber on the left hand side is the set of elements of $\mathfrak{k}_{x,r}^{G_s}/\mathfrak{k}_{x,r+}^{G_s}$ whose semisimple part is $\overline{g}_s$, while the fiber on the right hand side is the set of elements of $\mathfrak{k}_{x,r}/\mathfrak{k}_{x,r+}$ whose semisimple part is $\overline{g}_s,$ interpreted as an element of $\mathfrak{k}_{x,r}/\mathfrak{k}_{x,r+}.$ It thus suffices to show that any element $\overline{g}\in\mathfrak{k}_{x,r}/\mathfrak{k}_{x,r+}$ with semisimple part equal to $\overline{g}_s$ lies in $\mathfrak{k}_{x,r}^{G_s}/\mathfrak{k}_{x,r+}^{G_s}$. Jordan decomposition tells us that $\overline{g}$ commutes with $\overline{g}_s$. But Corollary \ref{finiteaffinecomparison2} then implies that $\overline{g}$ is contained in $\mathfrak{k}_{x,r}^{G_s}/\mathfrak{k}_{x,r+}^{G_s}$, as desired.
\end{proof}

\subsection{Conclusion of the proof}

The algebro-geometric input needed to upgrade Proposition \ref{bijectivity} to Theorem \ref{basecase} is the following.

\begin{proposition}\label{omegaetale}
The map $\Omega$ is \'etale.
\end{proposition}

\begin{proof}
Fix an element $s\in S.$ We start by showing that the map of smooth varieties
\begin{equation}\label{eq:auxiliaryetale}
(G^x/K_{x,0+})\times^{(N^x_s/K_{x,0+}^{G_s})}(\mathfrak{k}_{x,r}^{G_s}/\mathfrak{k}_{x,r+}^{G_s})\rightarrow(\mathfrak{k}_{x,r}/\mathfrak{k}_{x,r+})
\end{equation}
is \'etale along $(G^x/K_{x,0+})\times^{(N^x_s/K_{x,0+}^{G_s})}(\mathfrak{k}_{x,r}^{G_s}/\mathfrak{k}_{x,r+}^{G_s})^{=G_s,\on{gen}}$. As the set of points where (\ref{eq:auxiliaryetale}) is not \'etale is a $G^x$-invariant closed subset, it suffices to show that for any semisimple element $\overline{g}_s\subseteq(\mathfrak{k}_{x,r}^{G_s}/\mathfrak{k}_{x,r+}^{G_s})^{=G_s,\on{gen}},$ (\ref{eq:auxiliaryetale}) is \'etale at $1\times\overline{g}_s.$ Equivalently, we need to show that $\Omega$ induces an isomorphism of tangent spaces at $1\times\overline{g}_s$. But the induced map on tangent spaces is isomorphic to the map
\[
[\mathfrak{p}_x/\mathfrak{k}_{x,0+},\overline{g}_s]\oplus\mathfrak{k}^{G_s}_{x,r}/\mathfrak{k}^{G_s}_{x,r+}\rightarrow\mathfrak{k}_{x,r}/\mathfrak{k}_{x,r+},
\]
which is an isomorphism by part (3) of Lemma \ref{inverseconstruction}.

The base change of (\ref{eq:auxiliaryetale}) along $(\mathfrak{k}_{x,r}/\mathfrak{k}_{x,r+})^{=G_0}\rightarrow (\mathfrak{k}_{x,r}/\mathfrak{k}_{x,r+})$ can be identified with
\[
(G^x/K_{x,0+})\times^{(N^x_s/K_{x,0+}^{G_s})}((\mathfrak{k}_{x,r}^{G_s}/\mathfrak{k}_{x,r+}^{G_s})\times_{\mathfrak{k}_{x,r}/\mathfrak{k}_{x,r+}}(\mathfrak{k}_{x,r}/\mathfrak{k}_{x,r+})^{=G_0})\rightarrow (\mathfrak{k}_{x,r}/\mathfrak{k}_{x,r+})^{=G_0}
\]
and must remain \'etale along $(G^x/K_{x,0+})\times^{(N^x_s/K_{x,0+}^{G_s})}(\mathfrak{k}_{x,r}^{G_s}/\mathfrak{k}_{x,r+}^{G_s})^{=G_s,\on{gen}}$. The proposition is now a consequence of the following lemma.

\begin{lemma}
The variety $(\mathfrak{k}_{x,r}^{G_s}/\mathfrak{k}_{x,r+}^{G_s})^{=G_s,\on{gen}}$ is a scheme-theoretic connected component of $(\mathfrak{k}_{x,r}^{G_s}/\mathfrak{k}_{x,r+}^{G_s})\times_{\mathfrak{k}_{x,r}/\mathfrak{k}_{x,r+}}(\mathfrak{k}_{x,r}/\mathfrak{k}_{x,r+})^{=G_0}$.
\end{lemma}

\begin{proof}
As \'etale morphisms preserve reducedness, $(\mathfrak{k}_{x,r}^{G_s}/\mathfrak{k}_{x,r+}^{G_s})\times_{\mathfrak{k}_{x,r}/\mathfrak{k}_{x,r+}}(\mathfrak{k}_{x,r}/\mathfrak{k}_{x,r+})^{=G_0}$ must be reduced along $(\mathfrak{k}_{x,r}^{G_s}/\mathfrak{k}_{x,r+}^{G_s})^{=G_s,\on{gen}}$. It thus suffices to show that $(\mathfrak{k}_{x,r}^{G_s}/\mathfrak{k}_{x,r+}^{G_s})^{=G_s,\on{gen}}$ is both open and closed in the Zariski topology on $(\mathfrak{k}_{x,r}^{G_s}/\mathfrak{k}_{x,r+}^{G_s})\times_{\mathfrak{k}_{x,r}/\mathfrak{k}_{x,r+}}(\mathfrak{k}_{x,r}/\mathfrak{k}_{x,r+})^{=G_0}$. Closedness follows from the fact that $(\mathfrak{k}_{x,r}^{G_s}/\mathfrak{k}_{x,r+}^{G_s})^{=G_s,\on{gen}}$ is the intersection of $(\mathfrak{k}_{x,r}^{G_s}/\mathfrak{k}_{x,r+}^{G_s})\times_{\mathfrak{k}_{x,r}/\mathfrak{k}_{x,r+}}(\mathfrak{k}_{x,r}/\mathfrak{k}_{x,r+})^{=G_0}$ and $(\mathfrak{k}_{x,r}^{G_s}/\mathfrak{k}_{x,r+}^{G_s})^{\geq G_s}$.

To show openness, we can assume that $x\in\mathcal{A}(T)$, allowing us to use the dictionary with graded Lie algebras explained in Subsection \ref{ss:grading}. Let $\mathfrak{h}_{x,s}\subseteq\mathfrak{h}$ be the Lie subalgebra given by the direct sum of subspaces $\mathfrak{h}_{x,s,r}\cong\mathfrak{k}_{x,r}^{G_s}/\mathfrak{k}_{x,r+}^{G_s}\subseteq\mathfrak{k}_{x,r}/\mathfrak{k}_{x,r+}\cong\mathfrak{h}_r.$ It is clear that $\dim_k\mathfrak{h}_{x,s}=\dim_{k((t))}G_0$. We will show that an element $\overline{g}\in(\mathfrak{k}_{x,r}^{G_s}/\mathfrak{k}_{x,r+}^{G_s})\times_{\mathfrak{k}_{x,r}/\mathfrak{k}_{x,r+}}(\mathfrak{k}_{x,r}/\mathfrak{k}_{x,r+})^{=G_0}\subseteq\mathfrak{h}_{x,s,r}$ lies in $(\mathfrak{k}_{x,r}^{G_s}/\mathfrak{k}_{x,r+}^{G_s})^{=G_s,\on{gen}}$ iff $\on{ad}\overline{g}$ is invertible on $\mathfrak{h}/\mathfrak{h}_{x,s}.$ This will immediately imply openness.

By Lemma \ref{finiteaffinecomparison}, the kernel of $\on{ad}\overline{g}_{\on{ss}}$ has dimension $\dim_{k((t))}G_0=\dim_k\mathfrak{h}_{x,s}.$ Therefore, $\on{ad}\overline{g}$ acts invertibly on $\mathfrak{h}/\mathfrak{h}_{x,s}$ iff $\on{ad}\overline{g}$ acts nilpotently on $\mathfrak{h}_{x,s}.$ But this is equivalent to $\overline{g}$ lying in $(\mathfrak{k}_{x,r}^{G_s}/\mathfrak{k}_{x,r+}^{G_s})^{\geq G_s}$, which, as previously discussed, is equivalent to $\overline{g}$ lying in $(\mathfrak{k}_{x,r}^{G_s}/\mathfrak{k}_{x,r+}^{G_s})^{=G_s,\on{gen}}$, as desired.
\end{proof}
\end{proof}

\begin{proof}[Proof of Theorem \ref{basecase}]
Because $k$ is algebraically closed, Proposition \ref{bijectivity} implies that $\Omega$ is universally injective. 	Combining with Proposition \ref{omegaetale} and Theorem 41.14.1 of \cite[\href{https://stacks.math.columbia.edu/tag/025G}{Theorem 025G}]{stacks-project}, we see that $\Omega$ must be an open immersion. But a bijective open immersion is an isomorphism, as desired.
\end{proof}
\printbibliography
\end{document}